\documentclass{article}
\usepackage{amsmath}
\usepackage{amsfonts}
\usepackage{graphicx}

\newcommand{\betless}{\noalign{\vskip3pt plus 3pt minus 1pt}}
\newcommand{\bet}{\noalign{\vskip6pt plus 3pt minus 1pt}}

\renewcommand{\O }{\Omega }

\renewcommand{\lor }{\longrightarrow}

\newtheorem{theorem}{Theorem}
\newtheorem{lemma}{Lemma}
\newcounter{remark}
\def\theremark {\arabic{remark}}
\newenvironment{remark}{\refstepcounter{remark}\par\noindent{\bf Remark\ \theremark}\ }{\par}
\newtheorem{Proof}{Proof}

\newenvironment{proof}{\begin{Proof}\rm}{\hfill $\Box$ \end{Proof}}

\title{Static two-grid  mixed finite-element approximations to the Navier-Stokes equations}
\author{Javier de Frutos\thanks{Departamento de Matem\'{a}tica Aplicada,
Universidad de Valladolid. Spain. Research supported by Spanish MEC
under grant MTM2010-14919 and by JCyL under grant VA001A10-1 (frutos@mac.uva.es) } \and Bosco
Garc\'{\i}a-Archilla\thanks{Departamento de Matem\'{a}tica Aplicada
II, Universidad de Sevilla, Sevilla, Spain. Research supported by
Spanish MEC under grant MTM2009-07849 (bosco@esi.us.es)}
  \and Julia Novo\thanks{Departamento de
Matem\'aticas, Universidad Aut\'onoma de Madrid, Instituto de
Ciencias Matem\'aticas CSIC-UAM-UC3M-UCM, Spain. Research supported
by Spanish MEC under grant MTM2010-14919 (julia.novo@uam.es)}}
\begin{document}
\maketitle

\begin{abstract} A two-grid scheme based on mixed finite-element approximations to the incompressible Navier-Stokes
equations is introduced and analyzed. In the first level the
standard mixed finite-element approximation over a coarse mesh is
computed. In the second level the approximation is postprocessed by
solving a discrete Oseen-type problem on a finer mesh.
The
two-level method is optimal in the sense that, when a suitable value
of the coarse mesh diameter is chosen, it has
 the rate of convergence of
the standard mixed finite-element method over the fine mesh.
Alternatively, it can be seen as a postprocessed method in which the
rate of convergence is increased by one unit with respect to the
coarse mesh. The analysis takes into account the loss of regularity
at initial time of the solution of the Navier-Stokes equations in
absence of nonlocal compatibility  conditions.
 Some numerical experiments are shown.
\end{abstract}

\section{Introduction}

\label{sec:1} We consider the incompressible Navier--Stokes
equations
\begin{eqnarray}
\label{onetwo} u_t -\nu \Delta u + (u\cdot\nabla)u + \nabla p
&=& f,\\ \bet {\rm div}(u)&=&0,\nonumber
\end{eqnarray}
in a bounded domain $\Omega\subset {\mathbb R}^d$ ($d=2,3$)
with a smooth boundary subject to homogeneous Dirichlet
boundary conditions $u=0$ on~$\partial\Omega$.
In~(\ref{onetwo}), $u$ is the velocity field, $p$ the pressure,
$\nu>0$ the diffusion coefficient and~$f$ a given force field.

In this paper we study the following two-grid mixed
finite-element method for the spatial discretization of the
above equations. First, for the solution $(u,p)$ of the fully
nonlinear Navier-Stokes equations~(\ref{onetwo}) corresponding
to a given initial condition
\begin{eqnarray}\label{ic}
u(\cdot,0)=u_0,
\end{eqnarray}
the mixed finite-element approximation~$(u_H,p_H)$ over a coarse
mesh of diameter $H$ is computed. Then, for any time $t>0$, the
postprocessed approximation $(\tilde u_h,\tilde p_h)$  is obtained
as the mixed finite-element approximation over a finer mesh ($h<H$)
to the following steady Oseen-type problem:
\begin{equation}
\begin{array}{rcl}
\left.\begin{array}{r@{}} -\nu \Delta \tilde u
+(u_H({t})\cdot\nabla) \tilde u+\nabla \tilde
p=f-\frac{d}{dt}u_H({t})
\\
\betless {\rm div}(\tilde u)=0
\end{array}\right\}&\qquad \hbox{\rm in~$\Omega$},
\\
\bet &\hspace*{-41pt}\tilde u=0, \quad \hbox{~~~~~~~~~~\,\rm
on~$\partial\Omega$}.
\end{array}
\label{eq:oseen}
\end{equation}
In this paper we prove that, in terms of the fine mesh diameter $h$,
this two-grid technique is of optimal order in the sense that,
 for appropriate choices of the coarse mesh diameter $H$,
the method has the same rate of convergence of standard mixed finite
element approximations in the fine mesh. On the other hand, for a
suitable value of the discretization parameter $h$, the rate of
convergence of the postprocessed approximation in terms of $H$
increases by one unit the rate of convergence of the coarse standard
approximation. The improvement in precision is achieved in both the
$H^1$ norm for the velocity and the $L^2$ norm for the pressure in
the case of linear, quadratic and cubic elements. For other than
linear elements the rate of convergence in the $L^2$ norm of the
velocity is also increased by one unit. We remark that time
evolution is performed only at the coarse mesh whereas at the fine
grid the time appears only as a parameter (see equation
(\ref{eq:oseen})), thus the name of static two-grid method.

Two-grid or two-level methods are a well established technique for
nonlinear steady problems, see~\cite{Xu}. In~\cite{Lay-Tob},
\cite{Lay-Len} several two-level methods are considered to
approximate the steady Navier-Stokes equations. They require solving
a nonlinear system over a coarse mesh and, depending on
the algorithm chosen, one Stokes problem, one linear Oseen problem
or one Newton step over the fine mesh. The corresponding algorithms
 obtain the optimal rate of convergence in the fine mesh for
appropriate choices of the coarse mesh diameter $H$.

In the case of nonlinear evolutionary equations, two-grid techniques
have been proposed and studied in \cite{Abboud_etal},
\cite{Hou_etal}, \cite{Liu_Hou}, \cite{two-grid-nosotros}. In these
methods, as opposed to the method studied in the present paper, time
evolution is also performed over the fine mesh.  The advantage of
the method studied in the present paper is that since the time
integration is only carried out on the coarse mesh, computations on
the fine grid can be done at selected target time levels where an
improved approximation is desired, with the corresponding reduction
of computing time, specially if these target time levels are
sufficiently spaced in time. For this reason, although some of the
two-grid methods that incorporate the evolution in time of the fine
mesh approximation are more accurate, the method we present can
still be more efficient in terms of computational effort for a given
error level.


Two-grid techniques that integrate in time only on the coarse level
have previously been developed in \cite{Bosco-Julia-Titi},
\cite{Bosco-Julia-Titi-2} (see also \cite{Margolin_titi_wynne}) for
spectral methods, and later extended to mixed finite-element
formulations in \cite{mini}, \cite{bjj}, \cite{jbj_regularity}. In
all these works the two grid method is referred to as postprocessed
Galerkin method, and, instead of~(\ref{eq:oseen}), the
approximation~$(\tilde u_h,\tilde p_h)$ is found as an approximation
to the following Stokes problem
\begin{equation}
\begin{array}{rcl}
\left.\begin{array}{r@{}} -\nu \Delta \tilde u +\nabla \tilde
p=f-\frac{d}{dt}u_H({t}) -(u_H({t})\cdot\nabla) u_H({t})
\\
\betless {\rm div}(\tilde u)=0
\end{array}\right\}&\qquad \hbox{\rm in~$\Omega$},
\\
\bet &\hspace*{-41pt}\tilde u=0, \quad \hbox{~~~~~~~~~~\,\rm
on~$\partial\Omega$}.
\end{array}
\label{eq:stokes}
\end{equation}
This two-grid method will be termed  standard postprocessed method,
to differentiate it to that studied in the present paper, which will
be termed  new postprocessed method. Both, the standard and the new
postprocessed methods, have the same rate of convergence. However,
as already noted in~\cite{jbj_ima1} for nonlinear
convection-diffusion problems, the new postprocessing technique
produces more accurate approximations than the standard
postprocessed method, for moderate to small values of the diffusion
parameter $\nu$. This will also be the case in the numerical
experiments in the present paper for moderate values of the
Reynolds number.

In the present paper we take into account the loss of regularity
suffered by the solutions of the Navier-Stokes equations at the
initial time in the absence of nonlocal compatibility conditions.
Thus, for the analysis, we do not assume the solution~$u$ to have
more than second-order spatial derivatives bounded in $L^2$ up to
initial time $t=0$, since demanding further regularity requires the
data to satisfy nonlocal compatibility conditions unlikely to be
fulfilled in practical situations \cite{heyran0}, \cite{heyran2}.
Due to the loss of regularity at $t=0$,  the best error bound that
we can obtain is
$O(H^5|\log(H)|)$. For this reason we do not analyze
higher than cubic finite elements. The same limit in the rate of
convergence was found in \cite{heyran2} for standard mixed
finite-element approximations and in \cite{jbj_regularity},
\cite{two-grid-nosotros} for  two-grid schemes.

In practice, any method to numerically solve evolutionary equations
needs of some time discretization procedure. For brevity reasons, we
have preferred to present the method in a semidiscrete manner
without reference to any particular time discretization. However, we
emphasize that being static, the method we present can be applied
exactly in the same form,
 with any time discretization. The analysis of fully discrete procedures can
be developed along the same lines that appear in \cite{jbj_fully_discrete}, \cite{jbj_ima2}.

The rest of the paper is as follows. In Section 2 we introduce
some preliminaries and notation. In Section 3 we carry out the
error analysis of the new method. Finally, some numerical
experiments are shown in the last section.

\section{Preliminaries and notations}
\label{sec:2}
We will assume that
$\Omega$ is a bounded domain in~${\mathbb{R}}^{d},\, d=2,3$,
of class~${\cal C}^m$, for $m\ge 2$. When dealing with linear elements
($r=2$ below) $\Omega$ may also be a convex polygonal or polyhedral domain.
We consider
the Hilbert spaces
\begin{align*}
 H&=\left\{ u \in L^{2}(\O)^d \mid\mbox{div}(u)=0, \, u\cdot n_{|_{\partial \Omega}}=0 \right\},\\
V&=\left\{ u \in H^{1}_{0}(\O)^d \mid \mbox{div}(u)=0 \right\},
\end{align*}
 endowed
with the inner product of $L^{2}(\O)^{d}$ and $H^{1}_{0}(\O)^{d}$,
respectively. For  $l\ge 0$ integer and $1\le q\le \infty$, we consider the
standard  spaces, $W^{l,q}(\Omega)^d$, of functions with
derivatives up to order $l$ in $L^q(\Omega)$, and
$H^l(\Omega)^d=W^{l,2}(\Omega)^d$. We will denote by $\|\cdot \|_l$ the norm in $H^l(\Omega)^d$,
and~$\|\cdot\|_{-l}$ will represent
the norm of its dual space. We consider also the quotient spaces
$H^l(\Omega)/{\mathbb R}$ with norm $\| p\|_{H^l/{\mathbb R}}= \inf\{ \| p+c\|_l\mid  c\in {\mathbb R}\}$.

We recall the following Sobolev's imbeddings \cite{Adams}: For
$q \in [1, \infty)$, there exists a constant $C=C(\Omega, q)$ such
that
\begin{equation}\label{sob1}
\|v\|_{L^{q'}} \le C \| v\|_{W^{s,q}}, \,\,\quad
\frac{1}{q'}
\ge \frac{1}{q}-\frac{s}{d}>0,\quad q<\infty, \quad v \in
W^{s,q}(\Omega)^{d}.
\end{equation}
For $q'=\infty$, (\ref{sob1}) holds with $\frac{1}{q}<\frac{s}{d}$.

The following inf-sup condition is satisfied (see \cite{girrav}), there exists a constant $\beta>0$
such that
\begin{equation}\label{inf-sup}
\inf_{q\in L^2(\Omega)/{\Bbb R}}\sup_{ v\in H_0^1(\Omega)^d}\frac{(q,\nabla\cdot v)}{\|v\|_1\|q\|_{L^2/{\Bbb R}}}\ge \beta.
\end{equation}

Let $\Pi:
L^2(\O)^d \lor H$ be the $L^2(\O)^d $ projection onto $H$.
We denote by $A$ the Stokes
operator on $\O$: $$ A: \mathcal{D}(A)\subset H \lor H, \quad \, A=-\Pi\Delta , \quad
\mathcal{D}(A)=H^{2}(\O)^{d} \cap V. $$

We shall assume that $u$ is a strong solution up to time $t=T$, so that
\begin{equation}
\label{eq:M_1}
\|u(t)\|_1\le M_1,\quad
\|u(t)\|_2\le M_2,\quad 0\le t\le T,
\end{equation}
for some constants $M_1$ and~$M_2$. We shall also assume that there exists
a constant $\tilde M_2$ such that
\begin{equation}
\|f\|_1+\|f_t\|_1
+\|f_{tt}\|_1\le \tilde M_2,\quad 0\le t\le T.
\label{tildeM2}
\end{equation}
Finally, we shall assume that for some $k\ge 2$
$$
\sup_{0\le t\le T}
\bigl\| \partial_t^{\lfloor k/2\rfloor} f\bigr\|_{k-1-2{\lfloor k/2\rfloor}}+
\sum_{j=0}^{\lfloor (k-2)/2\rfloor}\sup_{0\le t\le T}
\bigl\| \partial_t^j f\bigr\|_{k-2j-2}
<+\infty,
$$
so that, according to Theorems~2.4 and~2.5 in~\cite{heyran0},
there exist positive constants $M_k$ and $K_{k}$ such that the
following bounds hold:
\begin{eqnarray}
\| u(t)\|_{k} + \| u_t(t)\|_{k-2}
+\| p(t)\|_{H^{k-1}/{\mathbb R}} \le M_k \tau(t)^{1-k/2},
\qquad\qquad\quad
\label{eq:u-inf}
\\
\quad
\int_0^t
\sigma_{k-3}(s)
\bigl(\| u(s)\|_{k}^2 + \| u_s(s)\|_{k-2}^2
+\| p(s)\|_{H^{k-1}/{\mathbb R}}^2+
\| p_s(s)\|_{H^{k-3}/{\mathbb R}} ^2\bigl)\,{\rm d}s\le
K_{k}^2,
\label{eq:u-int}
\end{eqnarray}
where $\tau(t)=\min(t,1)$ and~$\sigma_n=e^{-\alpha(t-s)}\tau^n(s)$ for some
$\alpha>0$. Observe that for $t\le T<\infty$, we can take~$\tau(t)=t$ and
$\sigma_n(s)=s^n$. For simplicity, we will take these values of
$\tau$ and~$\sigma_n$.

Let $\mathcal{T}_{h}=(\tau_i^h,\phi_{i}^{h})_{i \in
I_{h}}$, $h>0$ be a family of partitions of suitable domains $\Omega_h$, where
$h$ is the maximum diameter of the elements $\tau_i^h\in \mathcal{T}_{h}$,
and $\phi_i^h$ are the mappings of the reference simplex
$\tau_0$ onto $\tau_i^h$.

Let $r \geq 2$, we consider the finite-element spaces
$$
S_{h,r}=\left\{ \chi_{h} \in \mathcal{C}\left(\overline{\O}_{h}\right) \,  |
\, {\chi_{h}}{|_{\tau_{i}^{h}}}
\circ \phi^{h}_{i} \, \in \, P^{r-1}(\tau_{0})  \right\} \subset H^{1}(\O_{h}),
\
{S}_{h,r}^0= S_{h,r}\cap H^{1}_{0}(\O_{h}),
$$
where $P^{r-1}(\tau_{0})$ denotes the space of polynomials
of degree at most $r-1$ on $\tau_{0}$.

We shall denote by $(X_{h,r}, Q_{h,r-1})$
the so-called Hood--Taylor element \cite{BF,hood0}, when $r\ge 3$, where
$$
X_{h,r}=\left({S}_{h,r}^0\right)^{d},\quad
Q_{h,r-1}=S_{h,r-1}\cap L^2(\O_{h})/{\mathbb R},\quad r
\ge 3,
$$
and the so-called mini-element~\cite{Brezzi-Fortin91} when $r=2$,
where $Q_{h,1}=S_{h,2}\cap L^2(\O_{h})/{\mathbb R}$, and
$X_{h,2}=({S}_{h,2}^0)^{d}\oplus{\mathbb B}_h$. Here,
${\mathbb B}_h$ is spanned by the bubble functions $b_\tau$,
$\tau\in\mathcal{T}_h$, defined by
$b_\tau(x)=(d+1)^{d+1}\lambda_1(x)\cdots
 \lambda_{d+1}(x)$,  if~$x\in \tau$ and 0 elsewhere,
where  $\lambda_1(x),\ldots,\lambda_{d+1}(x)$ denote the
 barycentric coordinates of~$x$. For these elements a uniform inf-sup condition is satisfied
(see \cite{BF}), that is,
there exists a constant $\beta>0$ independent of the mesh grid size $h$ such that
\begin{equation}\label{lbbh}
 \inf_{q_{h}\in Q_{h,r-1}}\sup_{v_{h}\in X_{h,r}}
\frac{(q_{h},\nabla \cdot v_{h})}{\|v_{h}\|_{1}
\|q_{h}\|_{L^2/{\mathbb R}}} \geq \beta.
\end{equation}
The approximate velocity belongs to the discrete
divergence-free space
$$
V_{h,r}=X_{h,r}\cap \left\{ \chi_{h} \in H^{1}_{0}(\O_{h})^d \mid
(q_{h}, \nabla\cdot\chi_{h}) =0  \quad\forall q_{h} \in Q_{h,r-1}
\right\},
$$
which is not a subspace of $V$.

Let
$(u,p)\in  (H^2(\Omega)^d\cap V )\times ( H^1(\Omega)\slash \mathbb{R} )$
be the solution of a Stokes problem with right-hand side $g$,
we will denote by $s_h=S_{h}(u)\in V_{h}$ the so-called Stokes projection
(see \cite{heyran2}) defined as the velocity component of the solution
of the following problem:
find $(s_h,q_h)\in(X_{h,r},Q_{h,r-1})$ such that
\begin{align}
   \nu(\nabla s_h,\nabla\phi_h)+(\nabla q_h,\phi_h)&=(g,\phi_h)
&&\forall \phi_h\in X_{h,r},&&&&
\label{stokesnew}\\
(\nabla \cdot s_h,\psi_h)&=0&& \forall \psi_h\in  Q_{h,r-1}.&&&&
\label{stokesnew2}
\end{align}
The following bound holds for $2\le l\le r$:
\begin{equation}
\|u-s_h\|_0+h\|u-s_h\|_1\le C h^l \bigl (\|u\|_l+\|p\|_{H^{l-1}/{\mathbb R}}\bigr ).
\label{stokespro}
\end{equation}
The proof of (\ref{stokespro}) for $\Omega=\Omega_h$ can be found
in~\cite{heyran2}. The
bound for the pressure is \cite{girrav}
\begin{equation}\label{stokespre}
\|p-q_h\|_{L^2/{\mathbb R}}\le
C_\beta h^{l-1}\bigl (\|u\|_l+\|p\|_{H^{l-1}/{\mathbb R}}\bigr ),
\end{equation}
where the constant $C_\beta$ depends on the constant $\beta$ in the inf-sup condition
(\ref{lbbh}).

We consider the semi-discrete finite-element approximation~$(u_H,p_H)$ to~$(u,p)$,
solution
of~(\ref{onetwo})--(\ref{ic}).
That is,
given $u_H(0)=\Pi_Hu_{0}$, we compute
$u_{H}(t)\in X_{H,r}$ and $p_{H}(t)\in Q_{H,r-1}$, $t\in(0,T]$,   satisfying
\begin{align}\label{ten}
(\dot u_{H}, \phi_{H})
+     \nu( \nabla u_{H}, \nabla \phi_{H}) + b(u_{H}, u_{H}, \phi_{H}) +
( \nabla p_{H}, \phi_{H}) & =  (f, \phi_{H})&& \forall \, \phi_{H} \in X_{H,r},\\
(\nabla \cdot u_{H}, \psi_{H}) & = 0&&\forall \, \psi_{H} \in Q_{H,r-1},
\label{ten2}
\end{align}
where $b(u,v,w)=((u\cdot \nabla)v+\frac{1}{2}(\nabla \cdot u)v,w)$ for any $u,v,w\in H_0^1(\Omega)^d$.

For $2\le r\le 5$, provided that~(\ref{stokespro})--(\ref{stokespre}) hold for
$l\le r$,
and~(\ref{eq:u-inf})--(\ref{eq:u-int}) hold for $k=r$, then we have
\begin{equation}
\label{eq:err_vel(t)}
\| u(t)-u_H(t)\|_0+H\| u(t)-u_H(t)\|_1
\le C\frac{H^r}{t^{(r-2)/2}},
\quad 0\le t\le T,
\end{equation}
(see, e.g.,
\cite{jbj_regularity,heyran0,heyran2}),
and also,
\begin{equation}
\label{eq:err_pre(t)}
\| p(t)-p_H(t)\|_{L^2/{\mathbb R}} \le C\frac{H^{r-1}}{t^{(r'-2)/2}},
\quad 0\le t\le T,
\end{equation}
where $r'=r$ if $r\le 4$ and $r'=r+1$ if~$r=5$.
\section{The new postprocessed method}
The postprocessing technique we propose is a two-level or two-grid method. In the first level, we choose a coarse
mesh of size $H$ and compute the mixed finite-element approximation
~$(u_H,p_H)$ to~$(u,p)$ defined by (\ref{ten})-(\ref{ten2}). In the second level, the discrete velocity and pressure $(u_H(t),p_H(t))$
are postprocessed by solving the following linear Oseen problem: find $(\tilde u_h(t),\tilde p_h(t))\in (X_{h,r},Q_{h,r-1})$, $h<H$,
satisfying for all $\phi_h\in X_{h,r}$ and $\psi_h\in Q_{h,r-1}$
\begin{eqnarray}
\nu (\nabla \tilde u_h(t),\phi_h)+((u_H(t)\cdot \nabla)\tilde u_h(t),\phi_h)+(\nabla \tilde p_h(t),\phi_h)&=&(f(t)-\dot u_H(t),\phi_h),\qquad \label{pos11}\\
(\nabla \cdot \tilde u_h(t),\psi_h)&=&0.  \label{pos12}
\end{eqnarray}
Equations (\ref{pos11})-(\ref{pos12}) can also be solved over a higher order mixed finite-element space over the same grid. For
simplicity in the exposition we will only consider the case in which we refine the mesh at the postprocessing step.

Let us observe that projecting equation (\ref{pos11}) over the discretely-free space $V_{h,r}$, and
avoiding for simplicity the dependence on $t$ in the notation,  we get that $\tilde u_h\in V_{h,r}$
satisfies
\begin{equation}\label{pos_pro}
\nu (\nabla \tilde u_h,v_h)+((u_H\cdot \nabla)\tilde u_h,v_h)=(f-\dot u_H,\phi_h),\quad \forall v_h\in V_{h,r}.
\end{equation}
We now prove that equation (\ref{pos_pro}) is well-posed, i.e., for
$H$ small enough there exists a unique function $\tilde u_h\in
V_{h,r}$ solving (\ref{pos_pro}). Let us denote by $B^H$ the
bilinear form defined by \begin{equation}\label{BH}
 B^H(u_h,v_h)=\nu(\nabla
u_h,\nabla v_h)+((u_H\cdot \nabla) u_h,v_h),\quad u_h,v_h\in
V_{h,r}.
\end{equation}
We proceed to show that $B^H$ is coercive which implies that there exists a unique function $\tilde u_h\in V_{h,r}$ satisfying (\ref{pos_pro}).
Let us also observe that once a unique $\tilde u_h$ is found, using the inf-sup condition (\ref{lbbh}) one easily obtains
the existence and uniqueness of the pair $(\tilde u_h,\tilde p_h)$ satisfying (\ref{pos11})-(\ref{pos12}).

\begin{lemma}\label{lemacoer}
Let $B_H$ be the bilinear form defined in (\ref{BH}).  Then,  there exists a constant $C$ such that for
$t>0$ the following
bound holds:
\begin{equation}\label{coer_B^H}
\left|B^H(v_h,v_h)\right|\ge \left(\nu-C
\frac{H^{r-1+\gamma}}{t^{(r-2)/2}}\right)\|v_h\|_1^2,\quad \forall
v_h\in V_{h,r},
\end{equation}
where $\gamma=1/2$ if the dimension~$d$ is $d=2$, and $\gamma=1/4$ if $d=3$.
\end{lemma}
\begin{proof} To prove the coercivity of $B^H$ we follow \cite[p.
2042]{Lay-Tob}. Let us first observe that for any $v_h\in V_{h,r}$
$$
B^H(v_h,v_h)=\nu \|\nabla v_h\|_0^2-\frac{1}{2}(\nabla \cdot u_H,v_h\cdot v_h).
$$
Let $q_H$ be the $L^2$ orthogonal projection of $v_h\cdot v_h$ over $Q_{H,r-1}$,
so that applying standard finite-element theory \cite{Ciarlet} and interpolation theory
on Hilbert spaces (see e.\ g.~\cite[\S~II.2]{Temam} we have
$\|v_h\cdot v_h-q_H\|_{L^2(\Omega)/{\Bbb R}} \le CH^{\gamma}
\left\|v_h\cdot v_h\right\|_{\gamma}$, for $\gamma\in (0,1]$.
 Taking into account that the velocity $u$
satisfies $\nabla \cdot u=0$ then
\begin{equation}\label{pri_coer}
B^H(v_h,v_h)=\nu \|\nabla v_h\|_0^2-\frac{1}{2}(\nabla \cdot (u_H-u),v_h\cdot v_h-q_H).
\end{equation}
And then
$$
\left|(\nabla \cdot (u_H-u),v_h\cdot v_h-q_H)\right|\le C\|u_H-u\|_1\|v_h\cdot v_h-q_H\|_{L^2(\Omega)/{\Bbb R}}.
$$
Following \cite[p. 2042]{Lay-Tob} we get
\begin{equation}\label{law_h}
\|v_h\cdot v_h\|_\gamma\le C \|v_h\|_1^2,
\end{equation}
where $\gamma=1/2$ if $d=2$, and $\gamma=1/4$ if $d=3$.
Using (\ref{law_h}) together with (\ref{eq:err_vel(t)}) we get
$$
\left|(\nabla \cdot (u_H-u),v_h\cdot v_h-q_H)\right|\le C\frac{H^{r-1}}{t^{(r-2)/2}} H^\gamma\|v_h\|_1^2.
$$
Finally, going back to (\ref{pri_coer}) we reach (\ref{coer_B^H}).
\end{proof}
Let us observe that,  for $t>0$ and
$H<(t^{(r-2)/2}\nu/C)^{1/(r-1+\gamma)}$, as a consequence of
Lemma~\ref{lemacoer}, there exists a unique $\tilde u_h\in V_{h,r}$
satisfying (\ref{pos_pro}).

We introduce now a linearized problem that will be used in the proof
of Theorem 1 where we state the rate of convergence of the new
method. Let $u$ be the velocity in the solution $(u,p)$ of
(\ref{onetwo})-(\ref{ic}). We will denote by $(v,j)$ the solution of
the following linearized problem
\begin{eqnarray}\label{linearizado}
-\nu \Delta v+(u\cdot \nabla )v+\nabla j&=&d\\
{\rm div}(v)&=&0\nonumber
\end{eqnarray}
in the domain $\Omega$ subject to homogeneous Dirichlet boundary conditions. Let us observe that since the divergence
of $u$ is zero the bilinear form:
$$
B(v,w)=\nu (\nabla v,\nabla w)+((u\cdot \nabla)v,w),\quad v,w\in V.
$$
 associated to this problem is continuous and coercive. Since the solution $v\in V$ of (\ref{linearizado}) satisfies
 $$
 B(v,w)=(d,w),\quad \forall w\in V
 $$
 by the Lax-Milgram
theorem there exists a unique solution $v$. Due to (\ref{inf-sup}) there exists also a unique pressure $j$.

We will assume in the sequel that both problem (\ref{linearizado}) and its dual problem satisfy the regularity assumption
\begin{equation}\label{regu_dual}
\|v\|_2+\|j\|_{H^1(\Omega)/{\Bbb R}}\le C \|d\|_0.
\end{equation}
The regularity assumption (\ref{regu_dual}) can be proved by using the analogous regularity of the Stokes problem and a bootstrap
argument, see \cite[Remark 2.1]{Lay-Tob}.

In the following lemma we state the rate of convergence of the mixed finite-element approximation to the solution $(v,j)$ of (\ref{linearizado})
defined as follows: find $(v_h,j_h)\in (X_{h,r},Q_{h,r-1})$ such that
\begin{eqnarray}
\nu(\nabla v_h,\nabla \phi_h)+((u\cdot \nabla)v_h,\phi_h)+(\nabla j_h,\phi_h)&=&(d,\phi_h),\ \forall \phi_h\in X_{h,r},\label{apro_lin1}\\
(\nabla \cdot v_h,\psi_h)&=&0,\quad \ \ \quad\forall \psi_h\in Q_{h,r-1}\label{apro_lin2}.
\end{eqnarray}
\begin{lemma} \label{le:comparo} Let $(v,j)$ be the solution of (\ref{linearizado}) and let $(v_h,j_h)$ be its mixed finite-element approximation. Then, the
following bounds hold for $2\le l\le r$
\begin{eqnarray}\label{cota_vel_lin}
\|v-v_h\|_0+h\|v-v_h\|_1&\le& C h^l\left(\|v\|_l+\|j\|_{H^{l-1}/{\Bbb R}}\right),\\
\|j-j_h\|_{L^2/{\Bbb R}}&\le& C h^{l-1}\left(\|v\|_l+\|j\|_{H^{l-1}/{\Bbb R}}\right)\label{cota_pre_lin}.
\end{eqnarray}
\end{lemma}
\begin{proof}
Let us denote by $s_h=S_h(v)$ the Stokes projection of $v$. More precisely, $(s_h,q_h)\in (X_{h,r},Q_{h,r-1})$ will be the solution of
(\ref{stokesnew})-(\ref{stokesnew2}) with right-hand-side $g=d-(u\cdot \nabla)v$. Let us denote by $e_h=s_h-v_h$.
Then, from (\ref{apro_lin1}) and (\ref{stokesnew}) we get
\begin{eqnarray}\label{eq:error}
\nu(\nabla e_h,\nabla w_h)+((u\cdot \nabla)e_h,w_h)=((u\cdot \nabla)(s_h-v),w_h),\quad \forall w_h\in V_{h,r}.
\end{eqnarray}
Taking $w_h=e_h$ in (\ref{eq:error}) and using (\ref{sob1}) we get
$$
\nu \|e_h\|_1^2\le C\|u\|_{L^{2d/(d-1)}}\|s_h-v\|_1\|e_h\|_{L^{2d}}\le C\|u\|_{1/2}\|s_h-v\|_1\|e_h\|_1,
$$
so that
\begin{eqnarray}\label{cota:e_h}
\|e_h\|_1\le C\|s_h-v\|_1.
\end{eqnarray}
Since $\|v-v_h\|_1\le \|v-s_h\|_1+\|e_h\|_1$ applying (\ref{stokespro}) we conclude $h\|v-v_h\|_1$ is bounded by the righ-hand side of~(\ref{cota_vel_lin}). The bound (\ref{cota_pre_lin}) for the pressure
is readily obtained by means of the auxiliary value $k_h=q_h-j_h$.
Subtracting (\ref{apro_lin1}) from~(\ref{cota_vel_lin}) and applying the inf-sup condition (\ref{lbbh}) one easily gets
$$
\beta \|k_h\|_{L^2/{\Bbb R}}\le \nu \|e_h\|_1+C\|u\|_{1/2}\|s_h-v\|_1,
$$
so that due to~(\ref{cota:e_h}) and~(\ref{stokespro}) it follows
that $\|k_h\|_{L^2/{\Bbb R}}$ is bounded by the right-hand side
of~(\ref{cota_pre_lin}). Since $\|j-j_h\|_{L^2/{\Bbb R}}\le
 \|j-q_h\|_{L^2/{\Bbb R}}+\|k_h\|_{L^2/{\Bbb R}}$, applying (\ref{stokespre}) we finally
prove (\ref{cota_pre_lin}).

We are left with the task of proving the bound for the $L^2$ norm of the error in the velocity. We will argue by duality.
Let us observe that
\begin{equation}
\label{eq:dual}
\|e_h\|_0=\sup_{\varphi\in L^2\ \varphi\ne 0}\frac{|(e_h,\varphi)|}{\|\varphi\|_0}.
\end{equation}
Let us fix $\varphi\in L^2$ and let us denote by $(w,k)$ the solution of the linearized dual problem
\begin{equation}
\label{eq:law}
\begin{array}{rcl}
\left.
\begin{array}{r}
-\nu \Delta w-(u\cdot \nabla )w+\nabla k=\varphi,\\
{\rm div}(w)=0,\end{array}\right\}&\quad&\hbox{\rm in $\Omega$},\\
u=0,\quad\,\,&&\hbox{\rm on $\partial\Omega$}.
\end{array}
\end{equation}
As stated before we assume that this problem
satisfies the regularity assumption~(\ref{regu_dual}), so that
\begin{equation}\label{regu_dual2}
\|w\|_2+\|k\|_{H^1(\Omega)/{\Bbb R}}\le C \|\varphi\|_0.
\end{equation}
We will denote by $(w_h,k_h)\in (X_{h,r},Q_{h,r-1})$ the mixed finite-element approximations to $(w,k)$. Reasoning exactly as before
and applying (\ref{regu_dual2}) we obtain
\begin{eqnarray}\label{cota_vel_lin_du}
\|w-w_h\|_1&\le& C h\left(\|w\|_2+\|k\|_{H^{1}/{\Bbb R}}\right)\le C h\|\varphi\|_0,\\
\|k-k_h\|_{L^2/{\Bbb R}}&\le& C h\left(\|w\|_2+\|k\|_{H^1/{\Bbb R}}\right)\le C h\|\varphi\|_0\label{cota_pre_lin_du}.
\end{eqnarray}
Integrating by parts we reach
\begin{eqnarray*}
(e_h,\varphi)&=&\nu(\nabla e_h,\nabla w)+((u\cdot \nabla)e_h,w)-((\nabla \cdot e_h),k)\nonumber\\
&=&\nu(\nabla e_h,\nabla (w-w_h))+((u\cdot \nabla)e_h,w-w_h)-((\nabla \cdot e_h),k-k_h)\nonumber\\
&&\quad +\nu(\nabla e_h,\nabla w_h)+((u\cdot \nabla)e_h,w_h).
\end{eqnarray*}
And then, applying (\ref{cota_vel_lin_du}) and (\ref{cota_pre_lin_du}) we reach
\begin{eqnarray}
|(e_h,\varphi)|&\le& C\nu \|e_h\|_1 h \|\varphi\|_0+C \|u\|_{1/2}\|e_h\|_1 h\|\varphi\|_0+C\|e_h\|_1 C h\|\varphi\|_0\nonumber\\
&&\quad + |\nu(\nabla e_h,\nabla w_h)+((u\cdot \nabla)e_h,w_h)|.
\label{eq:laref}
\end{eqnarray}
Then, to conclude, it only remains to bound $|\nu(\nabla e_h,\nabla w_h)+((u\cdot \nabla)e_h,w_h)|$ which by (\ref{eq:error})
is equal to $|((u\cdot \nabla)(s_h-v),w_h)|$.
Let us decompose
$$
|((u\cdot \nabla)(s_h-v),w_h)|\le |((u\cdot \nabla)(s_h-v),w_h-w)|+|((u\cdot \nabla)(s_h-v),w)|.
$$
Then, integrating by parts in the last term
\begin{eqnarray*}
|((u\cdot \nabla)(s_h-v),w_h)|\le C\|u\|_{1/2}\|s_h-v\|_1\|w_h-w\|_1+|((u\cdot \nabla)w,s_h-v)|,
\end{eqnarray*}
and the bound for the first term on the right hand side above concludes by applying (\ref{stokespro}) and (\ref{cota_vel_lin_du}).
Finally, since
$$
|((u\cdot \nabla)w,s_h-v)|\le C\|u\|_{L^{2d/(d-1)}}\|\nabla w\|_{L^{2d}}\|s_h-v\|_0.
$$
Applying Sobolev inequality (\ref{sob1}) together with (\ref{regu_dual2}) and (\ref{stokespro}) we reach
$$
|((u\cdot \nabla)w,s_h-v)|\le C \|u\|_{1/2}\|\varphi\|_0 h^l\left(\|v\|_l+\|j\|_{H^{l-1}/{\Bbb R}}\right),
$$
so that the proof is finished.
\end{proof}
We now state some results that will be use to get the rate of convergence of the new postprocessed method. The proof of the
following lemma can be found in \cite[Lemma 4]{jbj_NS_apos} for the case $r=2$ and in \cite[Lemma 5.1]{jbj_regularity}
for $r=3,4$.
\begin{lemma}\label{le:z_t}
Let $(u,p)$ be the solution of {\rm
(\ref{onetwo})--(\ref{ic})} and let $u_H$ be the mixed finite-element approximation to $u$. Then, there exists a positive
constant $C$ such that
\begin{eqnarray}\label{temporal_menos1}
\|u_t(t)-\dot u_H(t)\|_{-1} &\le&  \frac{C}{t^{(r-1)/2}} H^{r}\left|\log(H)\right|^{r'}, \ t\in(0,T], \
r=2,3,4,\qquad\\\label{temporal_menos2} \|A^{-1}\Pi\left(u_t(t)-\dot u_H(t)\right)\|_0 &\le&
\frac{C}{t^{(r-1)/2}} H^{r+1}\left|\log(H)\right|, \ t\in(0,T], \ r=3,4,
\end{eqnarray}
where $r'=2$ when $r=2$ and $r'=1$ otherwise.
\end{lemma}
The proof of the following lemma can be found in \cite[p. 226]{jbj_regularity}.
\begin{lemma}\label{u_Hmenos1}
Let $(u,p)$ be the solution of {\rm
(\ref{onetwo})--(\ref{ic})} and let $u_H$ be the mixed finite-element approximation to $u$. Then, there exists a positive
constant $C$ such that
\begin{eqnarray}\label{eq:u_Hmenos1}
\|u(t)- u_H(t)\|_{-1} &\le&  \frac{C}{t^{(r-1)/2}} H^{r+1}\left|\log(H)\right|, \ t\in(0,T], \quad r=3,4.
\end{eqnarray}
\end{lemma}
We end this section with a theorem that states the rate of convergence of the new postprocessed method.
\begin{theorem}\label{teo}
Let $(u,p)$ be  the solution  of {\rm (\ref{onetwo})--(\ref{ic})} and
for $r=2,3,4$ let {\rm(\ref{eq:u-inf})--(\ref{eq:u-int})} hold with $k=r+2$. Then,
there exist a positive constant $C$ such that the new postprocessed approximation $(\tilde u_h(t),\tilde p_h(t))$ defined by (\ref{pos11})-(\ref{pos12})
satisfies the following bounds for $t\in (0,T]$ and $H$ small enough:
\begin{eqnarray}\label{velo_lin}
\|u(t)-\tilde u_h(t)\|_1&\le& C h+\frac{C}{t^{1/2}}H^2 | \log(H)|^2,\quad r=2,\\\label{velo_cuad_cub}
\|u(t)-\tilde u_h(t)\|_j&\le& \frac{C}{t^{(r-2)/2}}h^{r-j}+\frac{C}{t^{(r-1)/2}}H^{r+1-j}|\log(H)|,\ j=0,1, \ r=3,4.\quad\\
\label{pre_todos}
\|p(t)-\tilde p_h(t)\|_{L^2/{\Bbb R}}&\le& \frac{C}{t^{(r-2)/2}}h^{r-1}+\frac{C}{t^{(r-1)/2}}H^r|\log(H)|^{r'},\quad r=2,3,4,
\end{eqnarray}
where $r'=2$ for $r=2$ and $r'=1$ otherwise.
\end{theorem}
\begin{proof}
Let us consider the linearized problem (\ref{linearizado}) with right hand side $d=f-u_t$. Then, the solution $(v,j)$ of (\ref{linearizado})
is the solution $(u,p)$ of (\ref{onetwo})--(\ref{ic}). Let us denote by $(v_h,j_h)$ its mixed finite-element approximation, that is
the solution of~(\ref{apro_lin1})--(\ref{apro_lin2}). This
approximation satisfy the error bounds (\ref{cota_vel_lin}) and (\ref{cota_pre_lin}) for $l=r$.
Let us decompose $u-\tilde u_h=(u-v_h)+(v_h-\tilde u_h)$ and $p-\tilde p_h=(p-j_h)+(j_h-\tilde p_h)$. To bound the first terms in
these two decompositions we will apply (\ref{cota_vel_lin}) and (\ref{cota_pre_lin}). In the rest of the proof we deal with the
other two terms.

Let us denote by $e_h=v_h-\tilde u_h$. Subtracting (\ref{pos11})
from~(\ref{apro_lin1}) it is easy to see that $e_h$ satisfies
\begin{eqnarray*}
\nu(\nabla e_h,\nabla \phi_h)+((u_H\cdot \nabla)e_h,\phi_h)=(\dot u_H-u_t,\phi_h)+(((u_H-u)\cdot \nabla)v_h,\phi_h),
\end{eqnarray*}
for all $\phi_h\in V_{h,r}$.
Taking $\phi_h=e_h$ in the above equation and applying (\ref{coer_B^H}) we get that for $H<(t^{(r-2)/2}\nu/C)^{1/(r-1+\gamma)}$ there
exists a constant $C$ such that
$$
\|e_h\|_1\le C\bigl(\|u_t-\dot u_H\|_{-1}+\|u_H-u\|_1\|v_h-u\|_1+\|u_H-u\|_0\|u\|_{3/2}\bigr),
$$
and applying (\ref{temporal_menos1}) from Lemma~\ref{le:z_t}, (\ref{eq:err_vel(t)}) and (\ref{cota_vel_lin})  we get
\begin{equation}\label{e_h1}
\|e_h\|_1\le  \frac{C}{t^{(r-1)/2}} H^{r}\left|\log(H)\right|^{r'}+\frac{C}{t^{(r-2)/2}}
\bigl( h
H^{r-1}+H^r\bigr),
\end{equation}
from which (\ref{velo_lin}) and the case $j=1$ in (\ref{velo_cuad_cub}) are concluded.

We now get the error bound for the pressure. Let us denote $r_h=j_h-\tilde p_h$. Subtracting~(\ref{pos11})
from~(\ref{apro_lin1}) and using (\ref{lbbh}) it is easy to obtain
\begin{eqnarray*}
\beta \|r_h\|_{L^2/{\Bbb R}}&\le& \nu \|e_h\|_1+C\|u_H\|_{1/2}\|e_h\|_1+\|\dot u_H-u_h\|_{-1}\nonumber\\
&&\quad+C\|u_H-u\|_1\|v_h-u\|_1+C\|u_H-u\|_0\|u\|_{3/2},
\end{eqnarray*}
from which we get (\ref{pre_todos}) applying (\ref{e_h1}), (\ref{temporal_menos1}) from Lemma~\ref{le:z_t}, (\ref{eq:err_vel(t)}) and (\ref{cota_vel_lin}).

To conclude we get the error bound for the velocity in the $L^2$ norm. We will argue as in the proof of
Lemma~\ref{le:comparo}, that is, recalling~(\ref{eq:dual}), for $\varphi\in L^2(\Omega)$ we consider the
solution~$(w,k)$ of~(\ref{eq:law}), so that~(\ref{eq:laref}) holds, and we are left to estimate~$|\nu(\nabla
e_h,\nabla w_h)+((u\cdot \nabla)e_h,w_h)|$. It is easy to see that
\begin{eqnarray}\label{ultima}
\nu(\nabla e_h,\nabla w_h)+((u\cdot \nabla)e_h,w_h)=(\dot u_H-u_t,w_h)+(((u_H-u)\cdot \nabla)\tilde u_h,w_h).
\end{eqnarray}
Let us now bound the two terms on the right hand side of (\ref{ultima}).
For the first one, using (\ref{regu_dual2}) and (\ref{cota_vel_lin_du}) we get
\begin{eqnarray*}
(\dot u_H-u_t,w_h)&=&(\dot u_H-u_t,w_h-w)+(\dot u_H-u_t,w)\nonumber\\
&\le& \|\dot u_H-u_t\|_{-1}\|w_h-w\|_1+\|A^{-1}\Pi(\dot u_H-u_t)\|_0\|Aw\|_0\nonumber\\
&\le&C \|\dot u_H-u_t\|_{-1} h\|\varphi\|_0+C\|A^{-1}\Pi(\dot u_H-u_t)\|_0\|\varphi\|_0.
\end{eqnarray*}
Applying now (\ref{temporal_menos1}) and (\ref{temporal_menos2}) we have that $(\dot u_H-u_t,w_h)$ is
$O(H^{r+1}|\log(H)|/(t^{(r-1)/2})$ for $r=3,4$. Finally, we will bound the second term on the right hand side
of (\ref{ultima}). To this end we decompose
\begin{eqnarray*}
(((u_H-u)\cdot \nabla)\tilde u_h,w_h)&=&(((u_H-u)\cdot \nabla) (\tilde u_h-u),w_h)+(((u_H-u)\cdot \nabla) u,w_h)\nonumber\\
&\le&  C\|u_H-u\|_{1/2}\|\tilde u_h-u\|_1\|w_h\|_1+(((u_H-u)\cdot \nabla) u,w_h)\nonumber\\
&\le& C H\|\tilde u_h-u\|_1\|\varphi\|_0+(((u_H-u)\cdot \nabla) u,w_h),
\end{eqnarray*}
where in the last inequality we have applied (\ref{eq:err_vel(t)}) and we have bounded $\|w_h\|_1\le C\|w\|_1\le C\|\varphi\|_0$. Then, to conclude,
it only remains to bound $(((u_H-u)\cdot \nabla) u,w_h)$. Adding and subtracting $w$ we get
\begin{eqnarray*}
(((u_H-u)\cdot \nabla)u,w_h)&=&(((u_H-u)\cdot \nabla) u,w_h-w)+(((u_H-u)\cdot \nabla) u,w)\nonumber\\
&\le&C\|u_H-u\|_0\|u\|_{3/2}\|w_h-w\|_1+C\|u_H-u\|_{-1}\|\nabla  u \cdot w\|_1\nonumber\\
&\le& C\|u_H-u\|_0 h\|\varphi\|_0+C\|u_H-u\|_{-1}\|u\|_2\|w\|_2\nonumber\\
&\le& C\|u_H-u\|_0 h\|\varphi\|_0+C\|u_H-u\|_{-1}\|\varphi\|_0,
\end{eqnarray*}
where we have applied (\ref{regu_dual2}). To conclude we apply (\ref{eq:err_vel(t)}) and Lemma~\ref{u_Hmenos1}.
\end{proof}
\begin{remark}
We observe from Theorem~\ref{teo} that the postprocessed method increases the rate of convergence of the Galerkin
method in one unit in terms of $H$, the size of the coarse mesh. In the case of linear elements the improvement
is only achieved in the $H^1$ norm of the velocity but it is not obtained in the $L^2$ norm. Analogous results
had been obtained for the standard postprocessing in the linear case, see \cite{mini}, \cite{jbj_NS_apos}. Let us also observe
that a correct selection of the coarse and fine mesh diameters gives for the new postprocessed method the same rate of convergence than the Galerkin
method over the fine mesh, although, of course, with different constants in the error bounds. The advantage of the method we
propose is the saving in  computational effort.
For the method we propose the time integration is performed using the standard Galerkin method over the coarse mesh and
only at the final time we solve  one linearized Oseen-type problem over the fine mesh. Let us observe that, for example,
 the selection $H=h^{1/2}$ allows to get
for the new postprocessed method the rate of convergence of the fine mesh in the $H^1$ norm when using linear elements.
The selection $H=h^{2/3}$ allows to get the rate of convergence of the fine mesh in the $H^1$ norm when using quadratic elements,
the choice $H=h^{3/4}$ gives the rate of  convergence of the fine mesh in the $L^2$ norm also for quadratics and so on.

The reason why we have not carried out the error analysis for higher than cubic finite elements is that, as in the
papers \cite{heyran2} and \cite{jbj_regularity}, due to the loss of regularity at $t=0$ no better than $O(H^5|\log(H)|)$
error bounds can be proved.
\end{remark}
\section{Numerical experiments}
We consider the Navier-Stokes equations (\ref{onetwo}) in the domain
$\Omega=[0,1]\times[0,1]$ subject to homogeneous Dirichlet boundary
conditions. For the numerical experiments of this section we
approximate the equations using the mini-element
\cite{Brezzi-Fortin91} over a regular triangulation of $\Omega$
induced by the set of nodes $(i/N,j/N)$, $0\le i,j\le N$, where
$N=1/H$ is an integer. We study the spatially semi-discrete case.
Hence, in the time integration (with the trapezoidal rule)
sufficiently small time steps were taken so as to ensure that errors
arising from the spatial discretization were dominant. In the first
experiment we take the forcing term $f(t,x)$ such that the solution
of (\ref{onetwo})-(\ref{ic}) with $\nu=0.05$ is
\begin{eqnarray*}
u^1(x,y,t)&=&\pi t  \sin^2(\pi x)\sin(2\pi y),\nonumber\\
u^2(x,y,t)&=&-\pi t\sin^2(\pi y)\sin(2\pi x),\\
p(x,y,t)&=&20 t x^2 y\nonumber.
\end{eqnarray*}
When using the mini-element it has been observed and reported in the
literature (see for instance \cite{Verfurth1}, \cite{VerfurthSIAM},
\cite{Bank-Welfert2} \cite{Kim}, \cite{Pier1} and \cite{Pier2}) that
the linear part of the approximation to the velocity, $u_h^l$, is a
better approximation to the solution $u$ than $u_h$ itself. The
bubble part of the approximation is only introduced for stability
reasons and does not improve the approximation to the velocity and
pressure terms. For this reason in the numerical experiments of this
section we only consider the errors in the linear approximation to
the velocity. Also, following \cite{mini}, we postprocess only the
linear approximation to the velocity, i.e., we solve  problem
(\ref{pos11})-(\ref{pos12}) substituting $u_H$ and $\dot u_H$  by
$u_H^l$ and $\dot u_H^l$ respectively.  The finite element space at
the postprocessed step is the same mini-element defined over a
refined mesh of size $h$ small enough to capture the asymptotic rate
of convergence in the fine grid. The coarse and fine mesh sizes in
the experiments are $H=1/6$, $H=1/8$, $H=1/10$ and $H=1/20$ and
$h=1/20$, $h=1/26$, $h=1/32$ and $h=1/36$ respectively.  For the
postprocessed approximation we also keep only the linear part. We
apply the postprocessing step only once at time $t=0.5$.
\begin{figure}[h]
\hspace{-0.5cm}
\includegraphics[width=6cm]{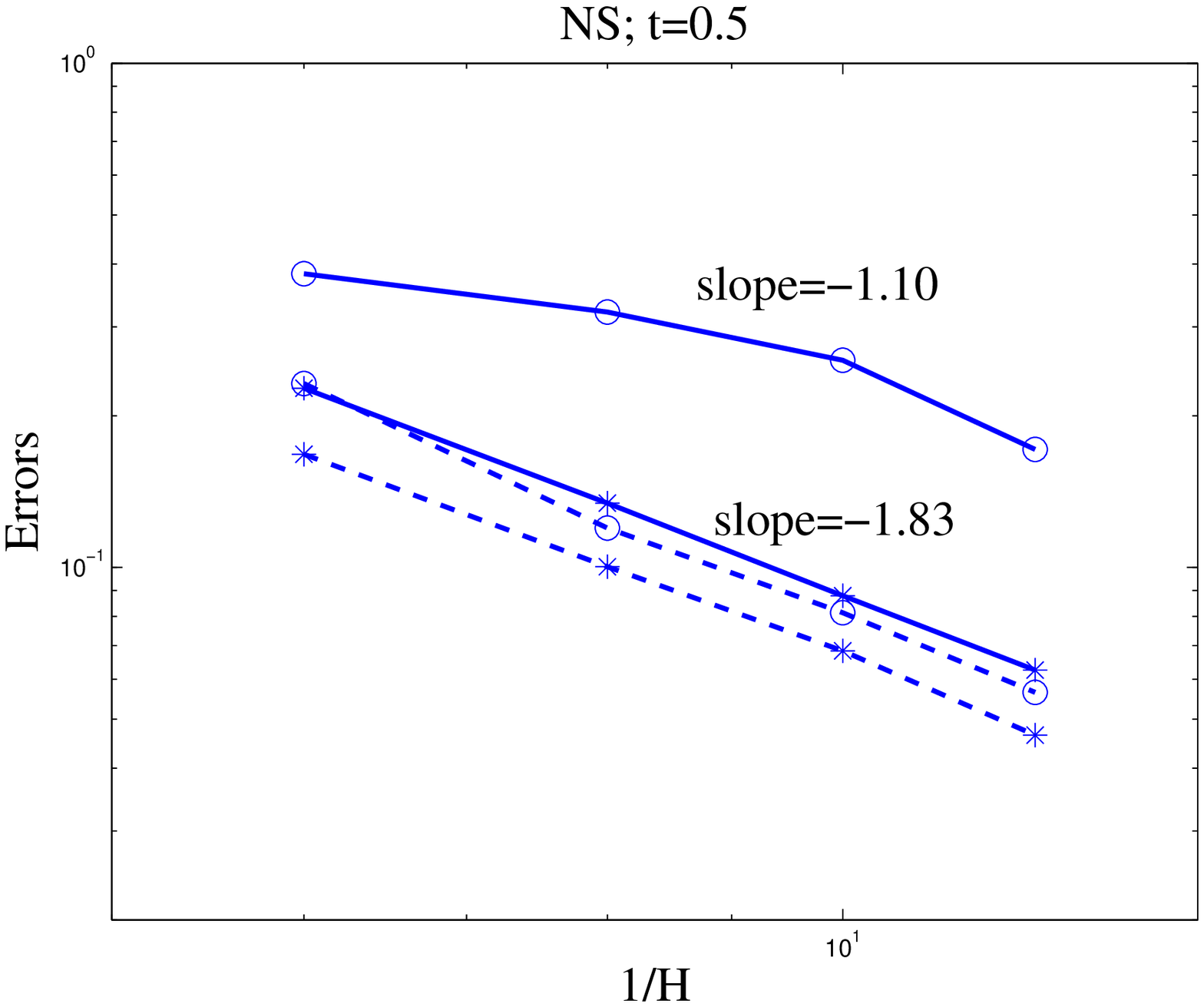}
\mbox{} \hfill
\includegraphics[width=6cm]{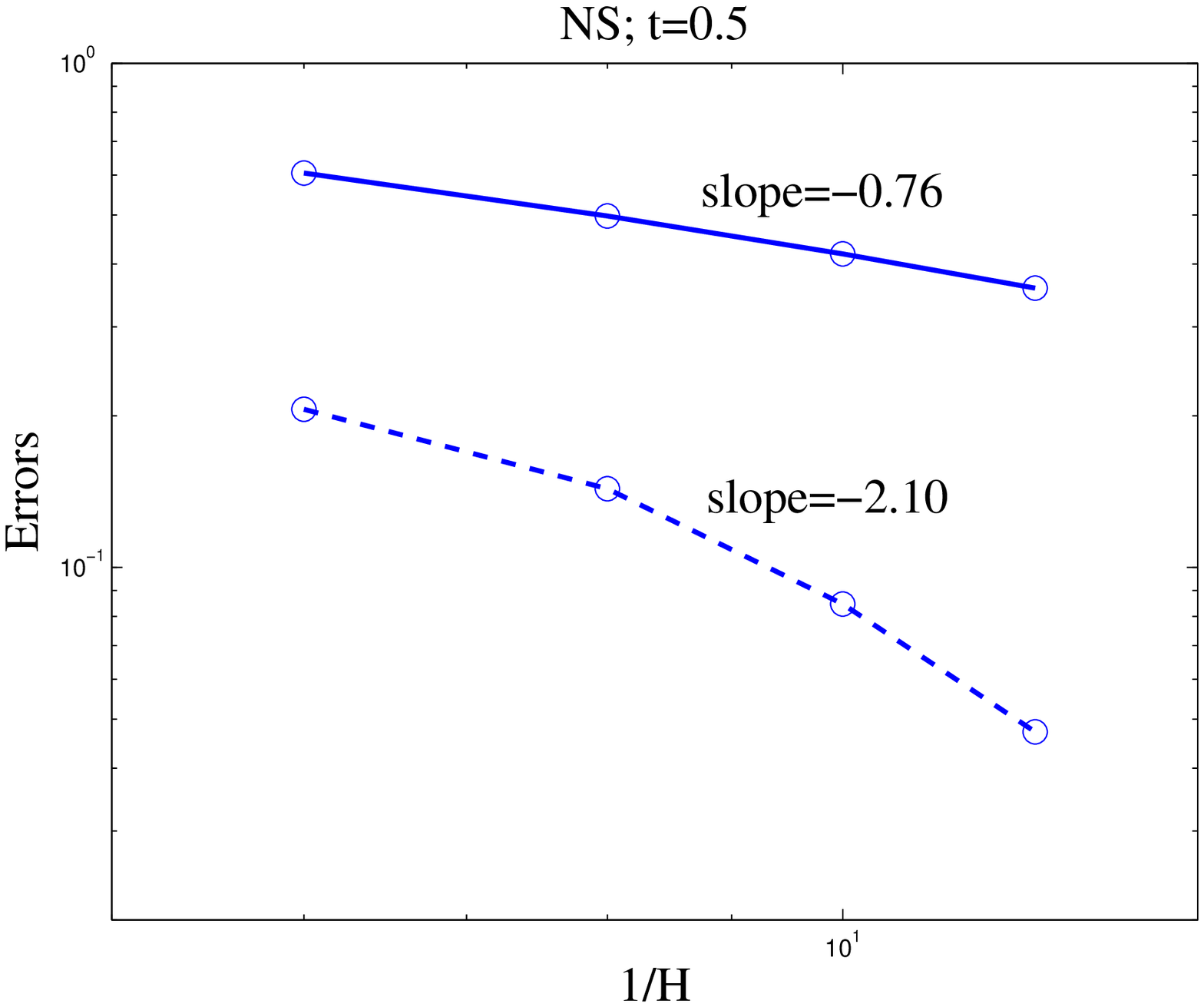}
\caption{Galerkin (solid lines) and postprocessed (dashed lines)
errors in $L^2$ (asterisks) and $H^1$ (circles) for $H=1/6$, $1/8$,
$1/10$ and $1/12$ and $h=1/20$, $1/26$, $1/32$ and $1/36$
respectively. On the left, errors for the first component of the
velocity. On the right, $L^2$ errors for the pressure.}\label{semi}
\end{figure}
In Figure~\ref{semi} we have represented the size of the Galerkin
and postprocessed errors with respect to the inverse of the coarse
mesh size $H$. On the left part of the picture we present the
results corresponding to the first component of the velocity. The
results obtained for the second component of the velocity are
analogous. On the right part of the picture we present the errors in
the pressure.
 In both pictures, we have used
solid line for the Galerkin method and dashed line for the
postprocessed method. The errors are measured in both the $H^1$ norm
and the $L^2$ norm. In the picture, they are represented by circles
($H^1$ norm errors) and asterisks ($L^2$ norm errors).
 We can observe on the left of
Figure~\ref{semi} that, in agreement with the theory, the
postprocessed method using the mini-element does not increase the
rate of convergence in the $L^2$ norm of the velocity although the
size of the errors are reduced. In the $H^1$ norm, however, also as
predicted by the theory, the postprocessed method does increase the
order of convergence by one unit (indeed, the errors of the
postprocessed method in the $H^1$ norm are slightly smaller than
those of the Galerkin method in the $L^2$ norm). The same
improvement is observed for the $L^2$ errors of the pressure on the
right of Figure~\ref{semi}. This means that we can obtain the level of
error corresponding to the fine mesh at essentially the cost of the
computation in the coarse mesh because the computation on the fine
mesh is performed only once at time $t=0.5$. Then, the dominant
computational cost is caused by the time evolution in the coarse
mesh saving time when compared with the time evolution in the fine
mesh that is needed in a standard approach.

\begin{figure}[h]
\hspace{-0.5cm}
\includegraphics[width=6cm]{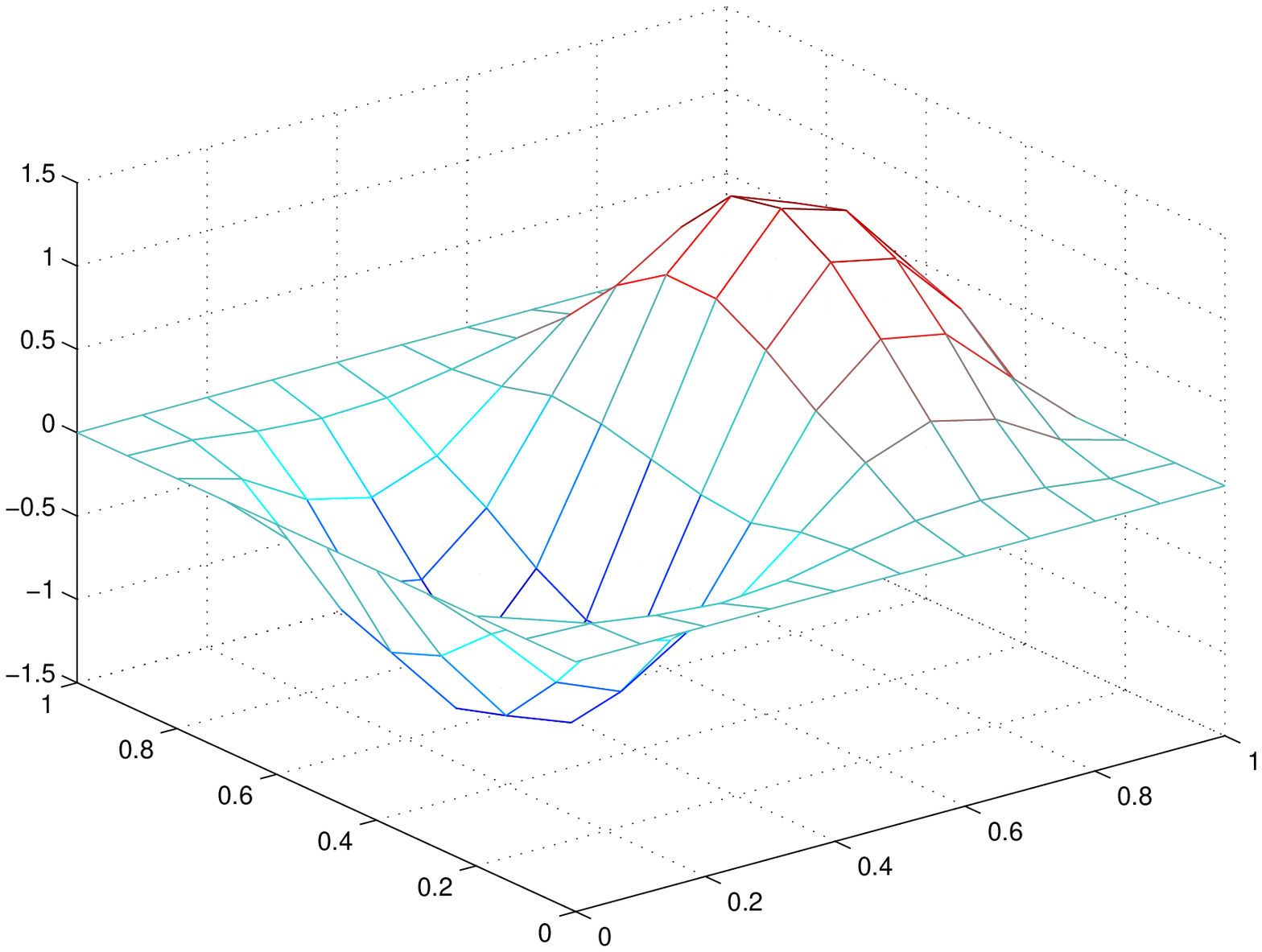}
\mbox{} \hfill
\includegraphics[width=6cm]{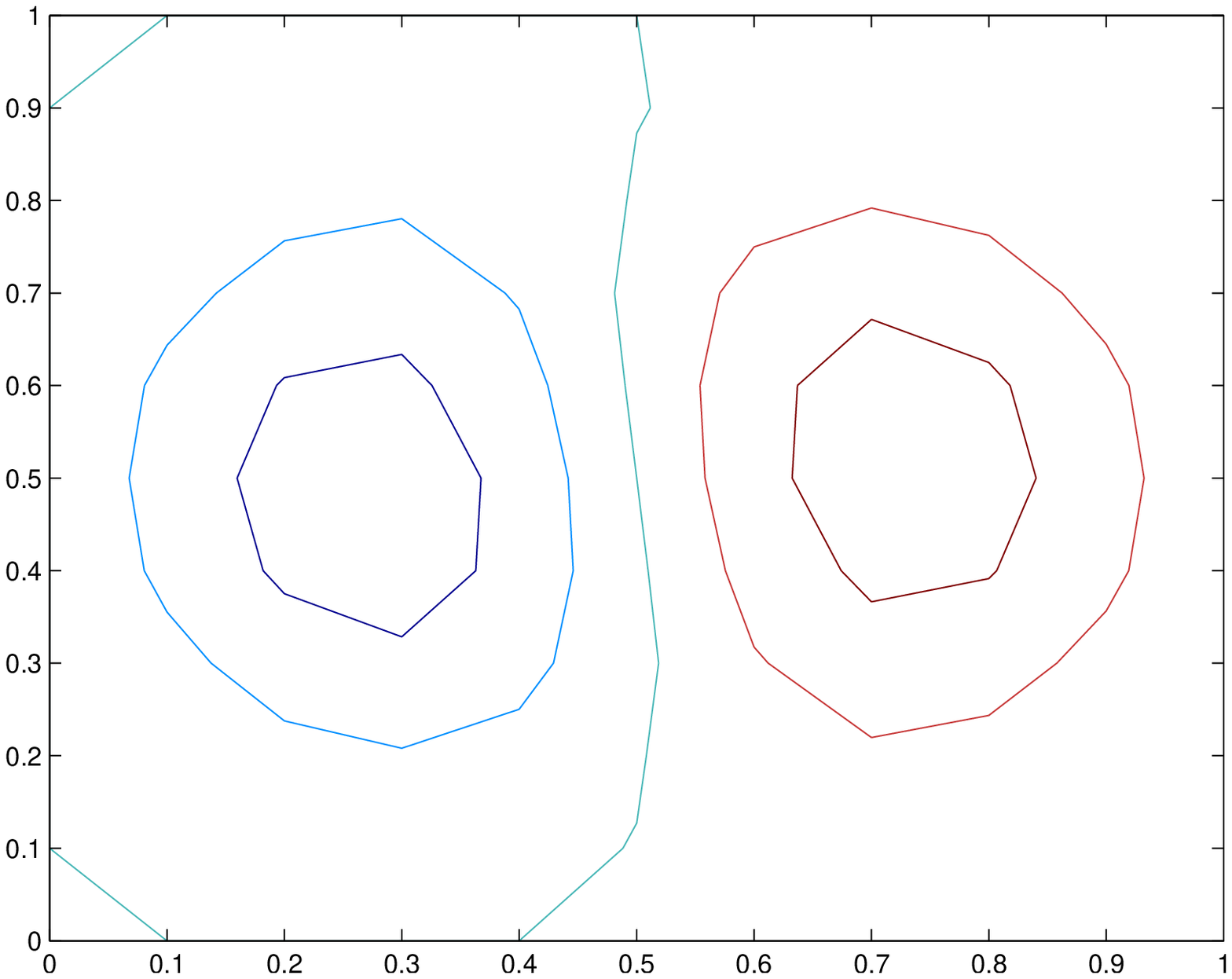}
\caption{First component of the velocity for the Galerkin method
with $\nu=0.01$ and $H=1/10$.}\label{gal1}
\end{figure}
\begin{figure}[h]
\hspace{-0.5cm}
\includegraphics[width=6cm]{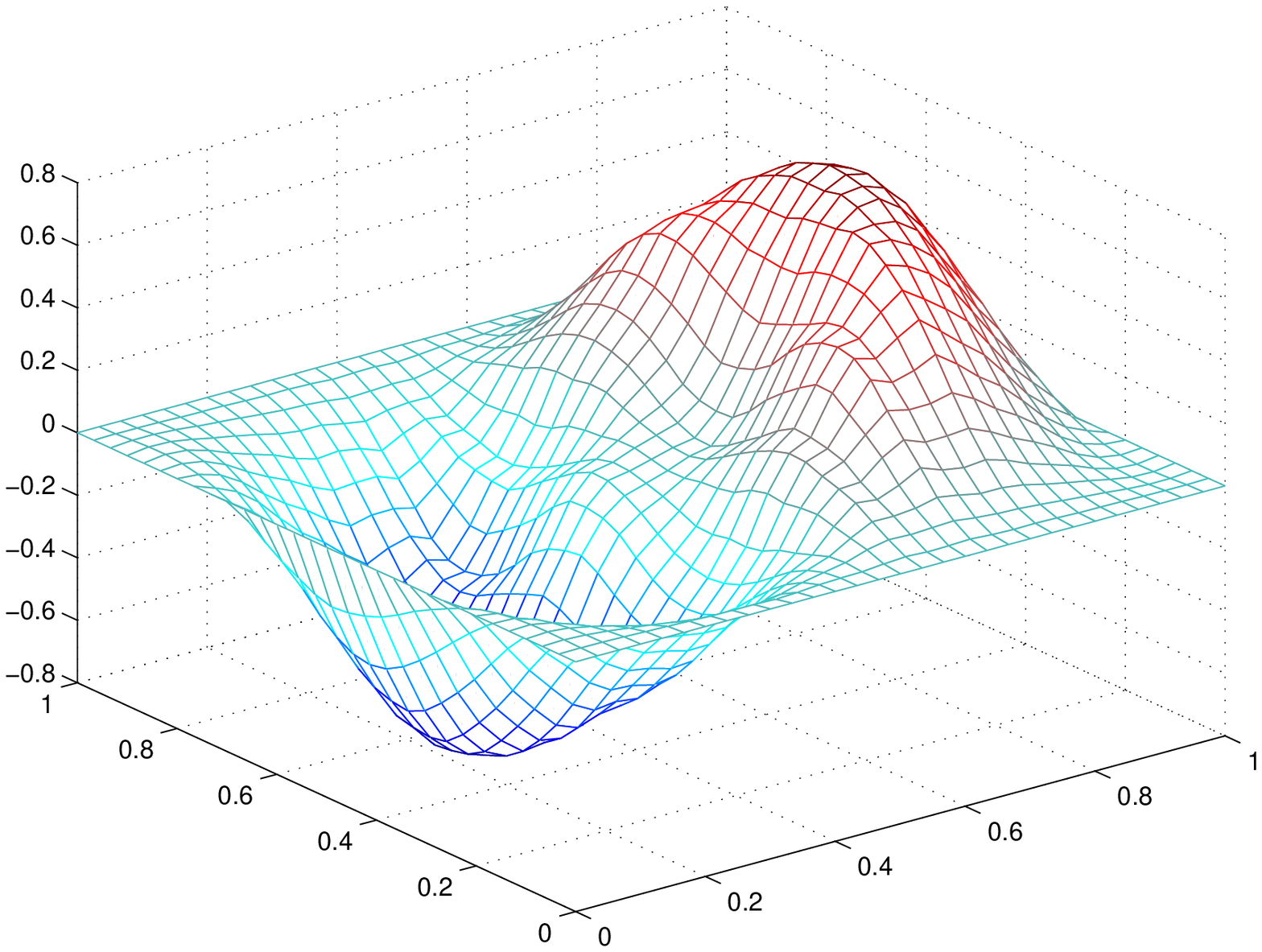}
\mbox{} \hfill
\includegraphics[width=6cm]{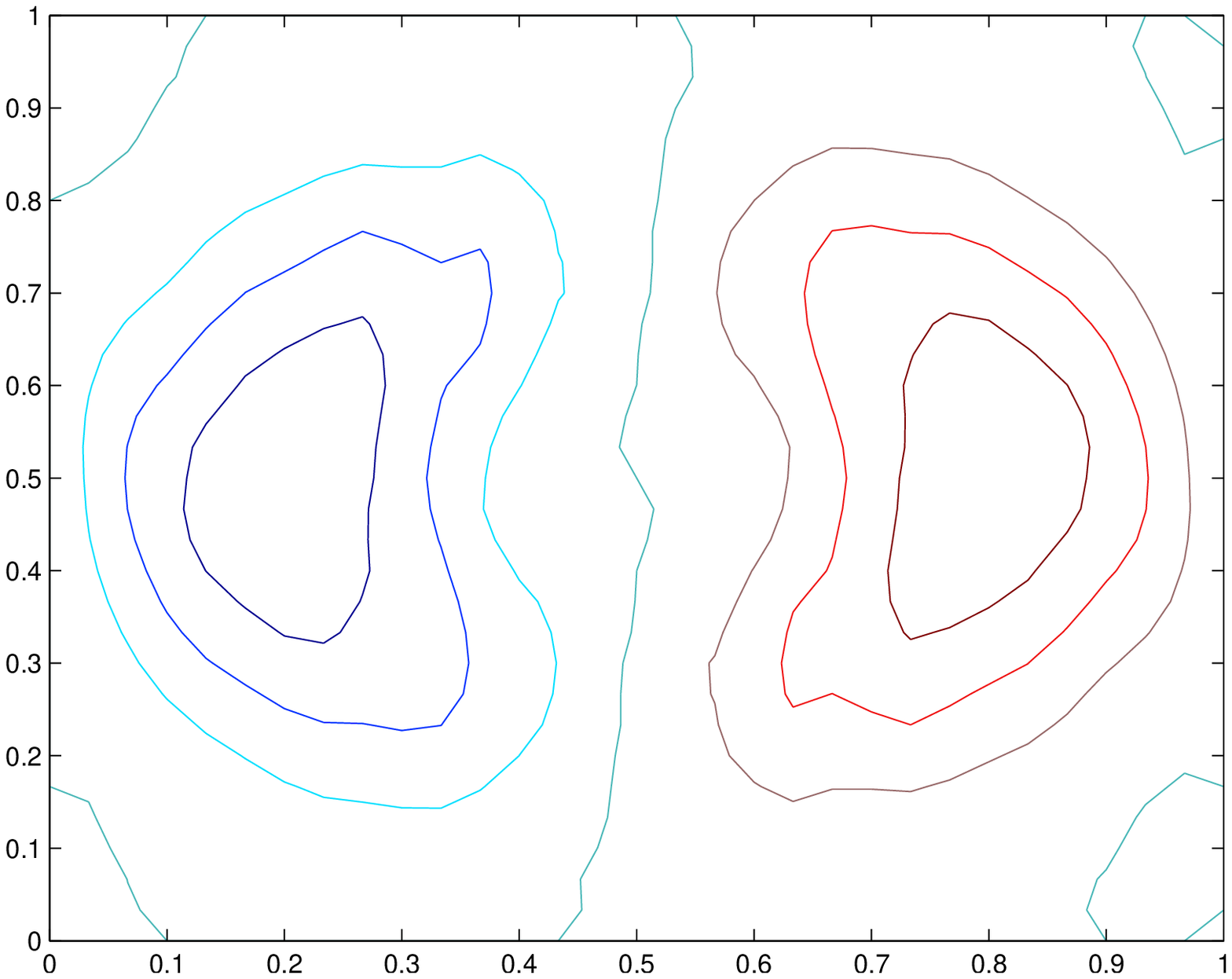}
\caption{First component of the velocity for the postprocessed
method with $\nu=0.01$, $H=1/10$ and $h=1/30$.}\label{pos1}
\end{figure}
\begin{figure}[h]
\hspace{-0.5cm}
\includegraphics[width=6cm]{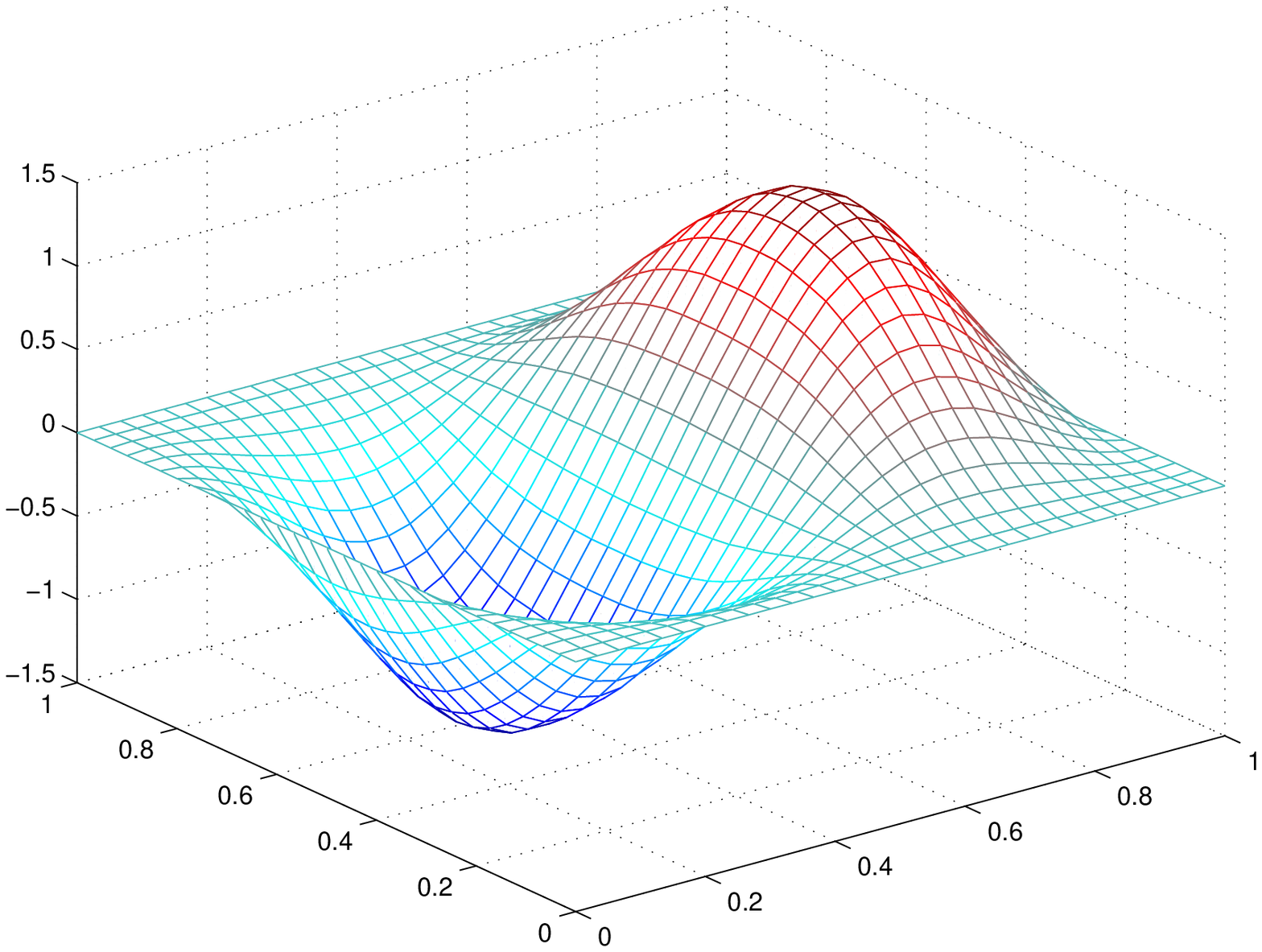}
\mbox{} \hfill
\includegraphics[width=6cm]{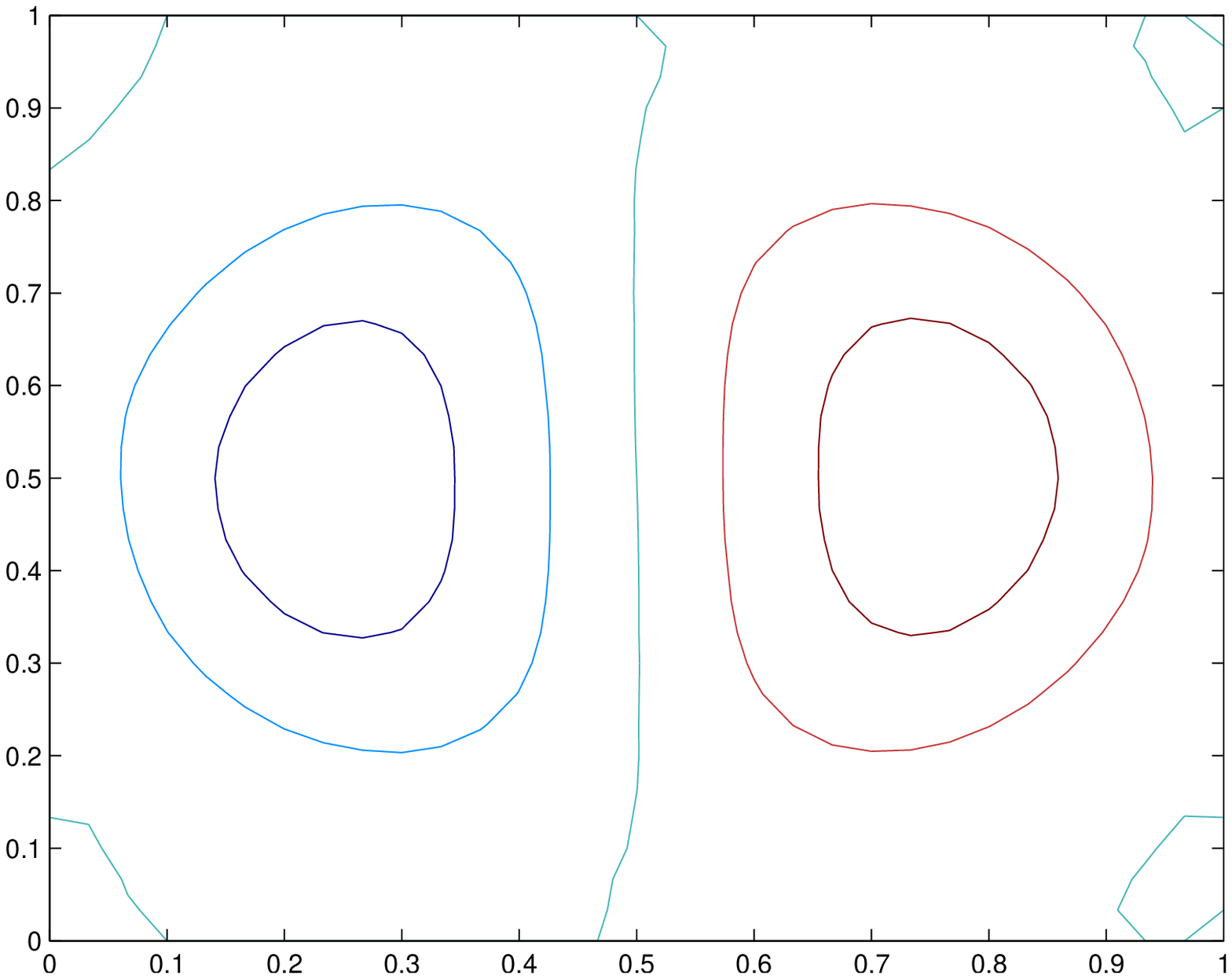}
\caption{First component of the velocity for the new postprocessed
method with $\nu=0.01$, $H=1/10$ and $h=1/30$.}\label{posnew1}
\end{figure}

\begin{figure}[h]
\hspace{-0.5cm}
\includegraphics[width=6cm]{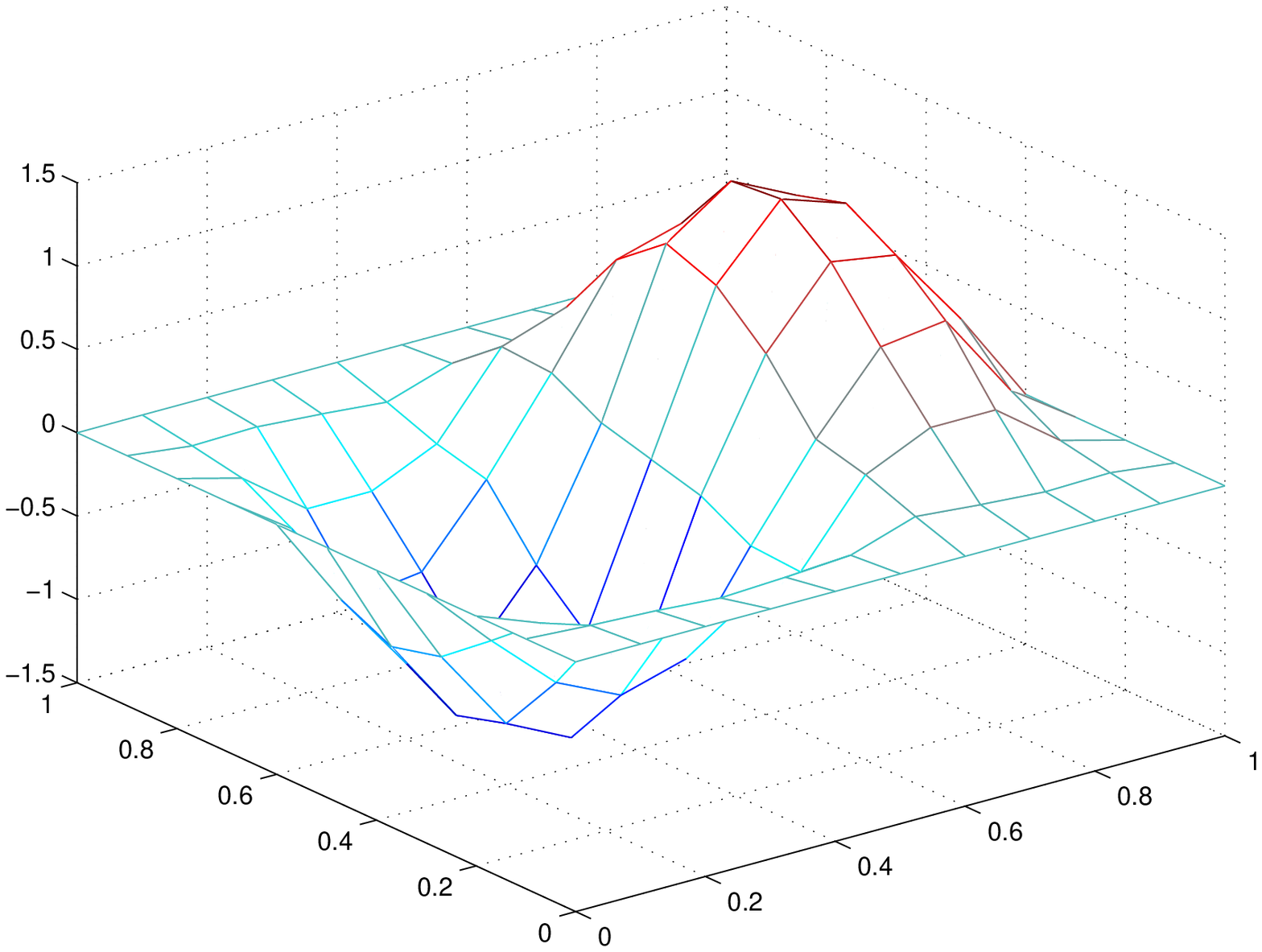}
\mbox{} \hfill
\includegraphics[width=6cm]{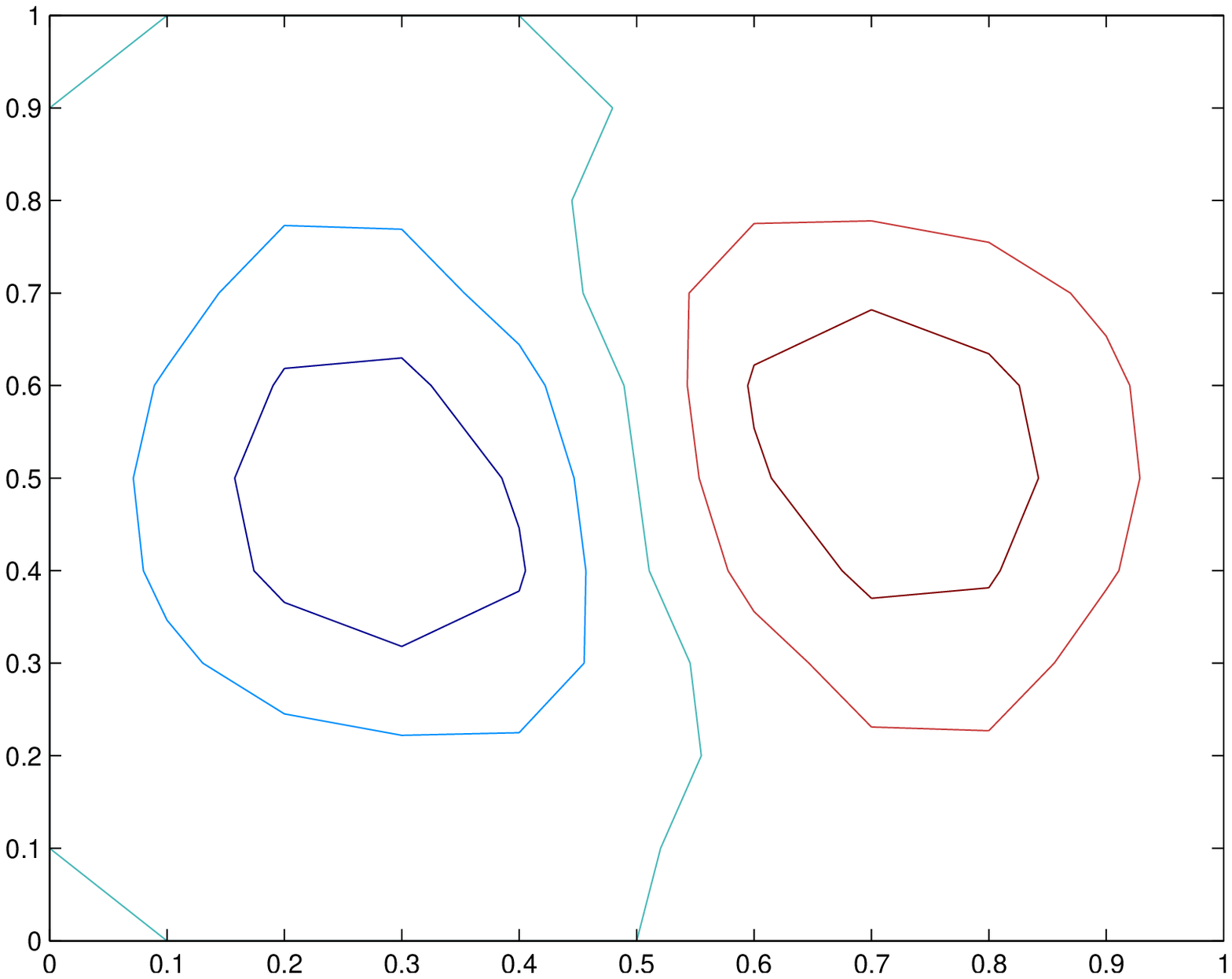}
\caption{First component of the velocity for the Galerkin method
with $\nu=0.005$ and $H=1/10$.}\label{gal2}
\end{figure}
\begin{figure}[h]
\hspace{-0.5cm}
\includegraphics[width=6cm]{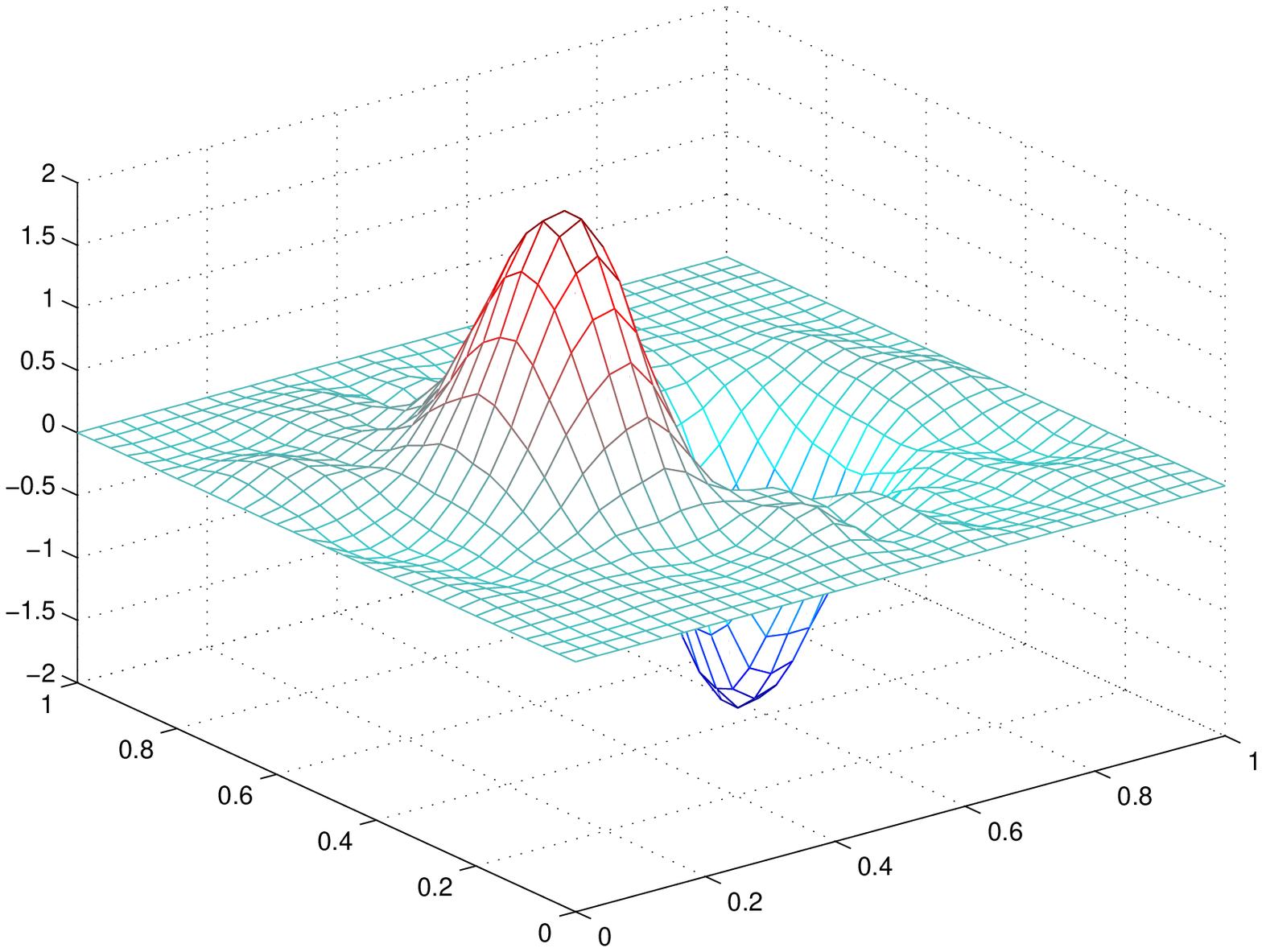}
\mbox{} \hfill
\includegraphics[width=6cm]{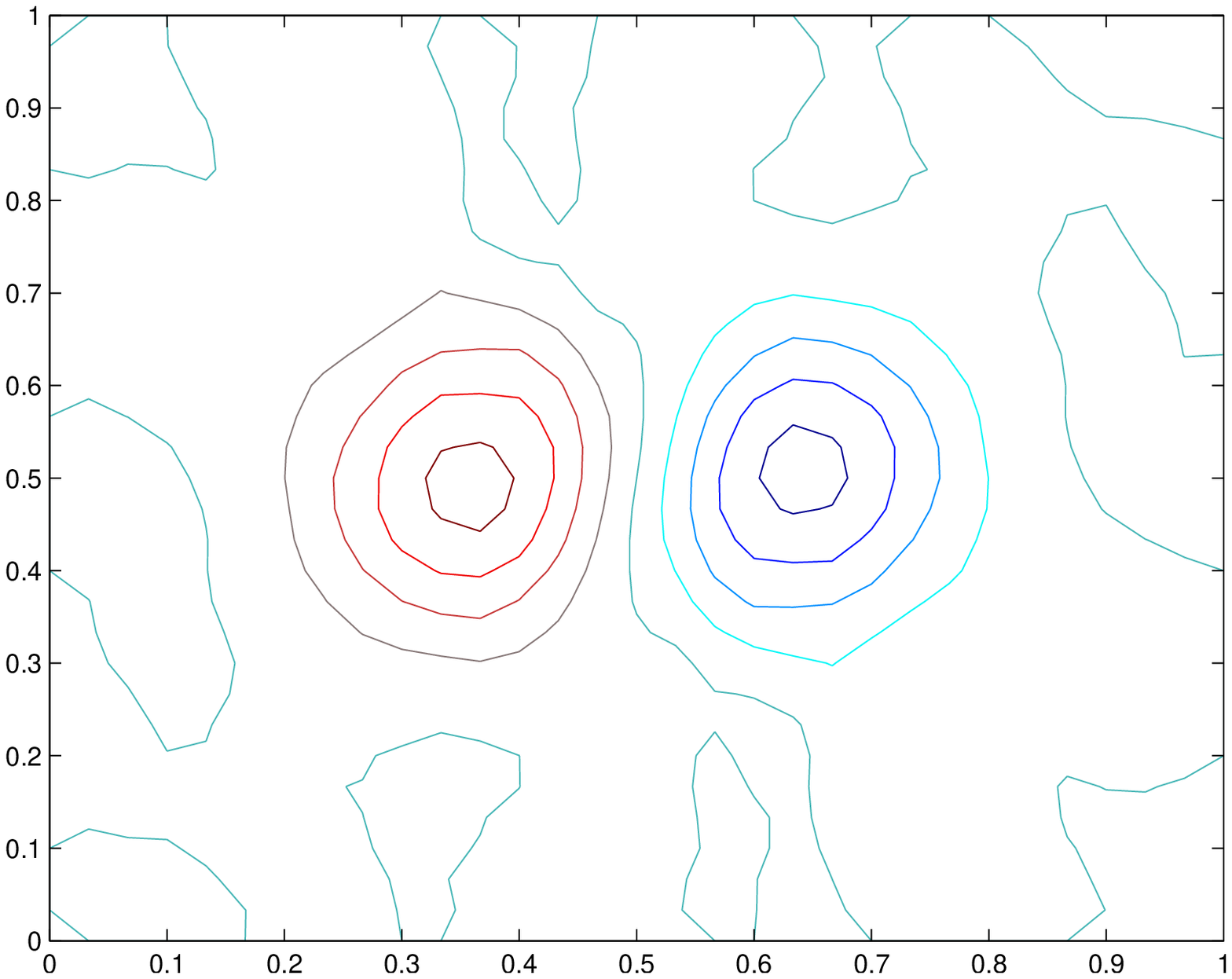}
\caption{First component of the velocity for the postprocessed
method with $\nu=0.005$, $H=1/10$ and $h=1/30$.}\label{pos2}
\end{figure}
\begin{figure}[h]
\hspace{-0.5cm}
\includegraphics[width=6cm]{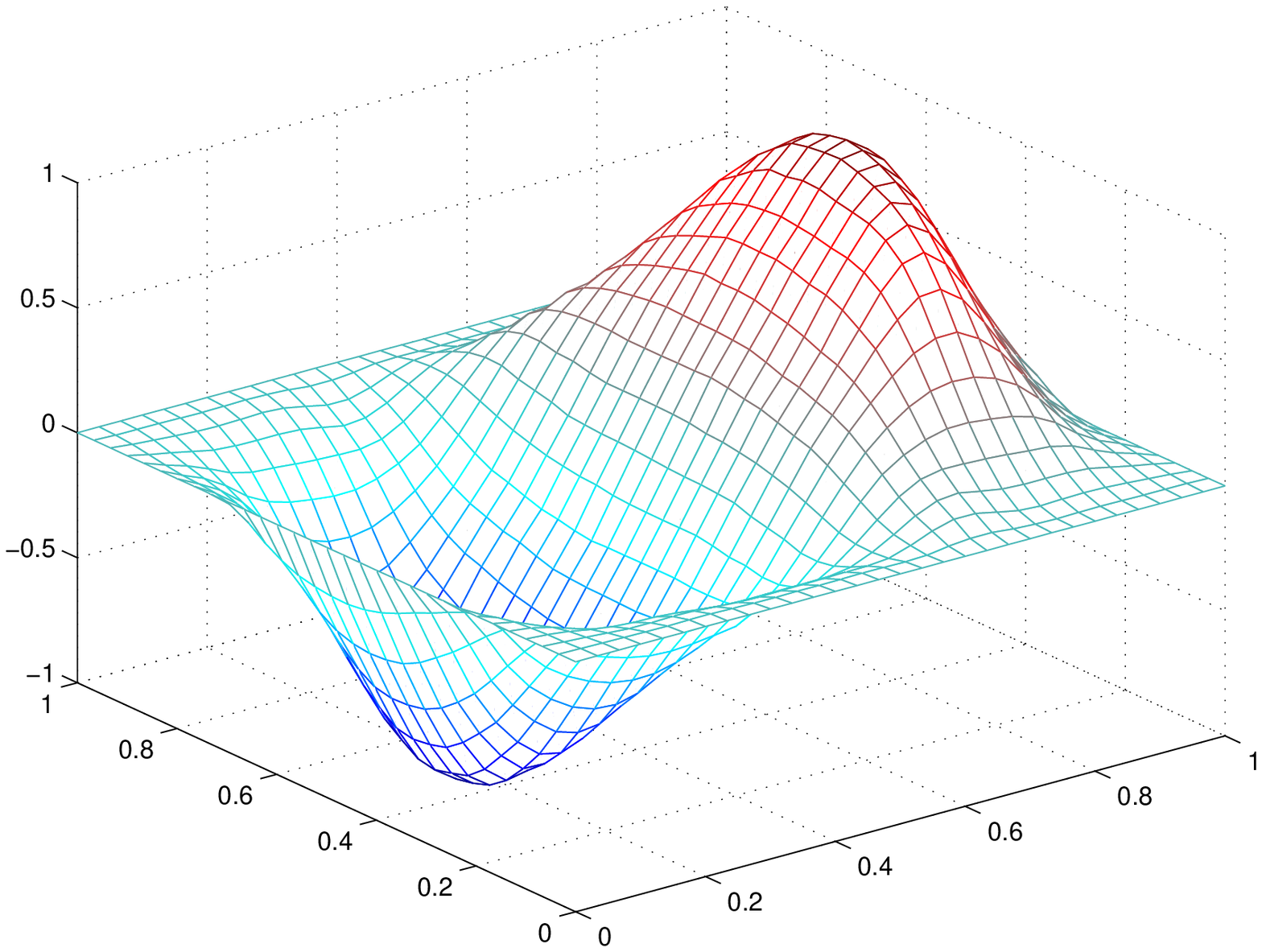}
\mbox{} \hfill
\includegraphics[width=6cm]{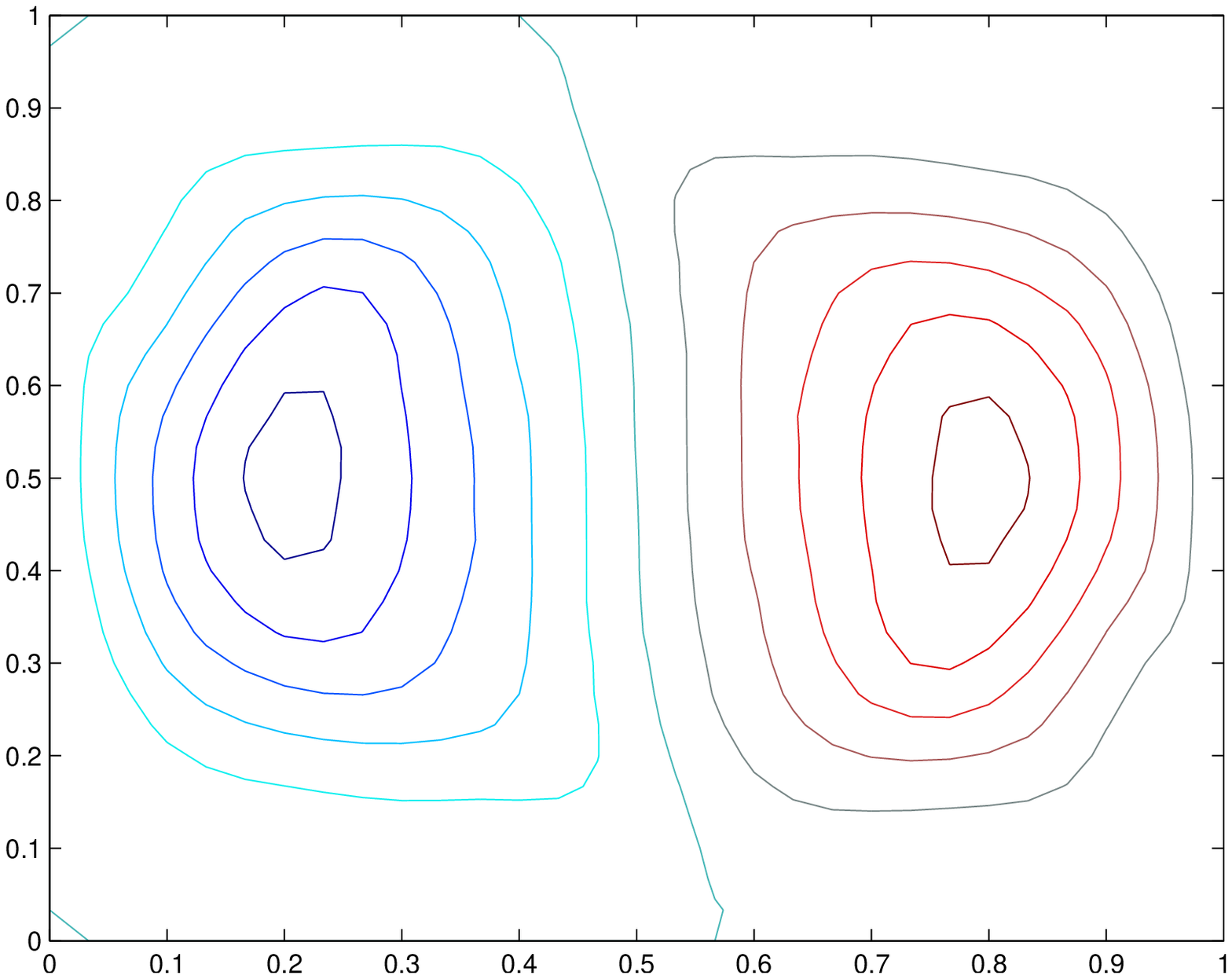}
\caption{First component of the velocity for the new postprocessed
method with $\nu=0.005$, $H=1/10$ and $h=1/30$.}\label{posnew2}
\end{figure}
In the next experiment we will show that the new postprocessed
method produces better results than both the Galerkin method and the
standard postprocessed method~(\ref{eq:stokes}). We consider now
equations (\ref{onetwo}) with initial condition
\begin{eqnarray*}
u^1(x,y,t)&=&-6\sin(\pi x)^3\sin(\pi y)^2\cos(\pi y),\nonumber\\
u^2(x,y,t)&=&6\sin(\pi x)^2\sin(\pi y)^3\cos(\pi x),
\end{eqnarray*}
and forcing term $f=0$.  We take first $\nu=0.01$. In
Figure~\ref{gal1} we have represented the linear part of the first
component of the velocity for the Galerkin method with $H=1/10$ at
time $T=0.5$.
In Figure~\ref{pos1} we show the standard postprocessed
approximation with $H=1/10$ and $h=1/30$. We can observe that the
standard postprocessing introduces some oscillations that were not
present in the Galerkin approximation. These oscillations are not
reduced with a smaller value of $h$. Finally, in
Figure~\ref{posnew1} we have represented the linear part of the
first component of the velocity for the new postprocessed
approximation and the same values of coarse and fine mesh sizes,
$H=1/10$ and $h=1/30$. We observe that this approximation does not
oscillate at all and it improves the accuracy of both Galerkin and
standard postprocessed approximations.

In the last experiment we repeat the experiment with a
smaller value of the diffusion parameter, $\nu=0.005$, and the same
values of $H$ and $h$ as that of the previous experiment.
 As it was already observed in
the case of convection-diffusion equations~\cite{jbj_ima1} the
behavior of the standard postprocessed method deteriorates as the
diffusion parameter decreases.

We can observe that both the Galerkin and the new postprocessed
approximations, see Figures~\ref{gal2} and \ref{posnew2} respectively,
do not present oscillations. As before, we can also observe the
smoothing effect achieved by postprocessing with the new method
proposed in this paper. On the other hand, the standard
postprocessed method produces a completely wrong approximation, see
Figure~\ref{pos2}. Let us remark that, as it has been noted before
in the literature, see \cite{Bretal}, \cite{simo}, the bubble
functions used in the mini-element to generate a stable mixed
finite-element satisfying the inf-sup condition (\ref{lbbh}) have
also a slightly stabilizing (over-diffusive) effect for moderate
values of the Reynolds number. This fact explains the
non-oscillating behavior of the linear part of the approximation to
the velocity in the Galerkin method of Figure~\ref{gal2}
We can also observe, see Figure~\ref{posnew2}, that
the over-diffusive effect appearing in the Galerkin approximation of
Figure~\ref{gal2} is attenuated by postprocessing with the new
method.

\end{document}